\documentclass[9pt]{amsart}
\usepackage{amssymb}
\usepackage{comment}
\usepackage{hyperref}
\usepackage[all]{xy}
\usepackage{enumerate}
\usepackage{ marvosym }
\usepackage{ bbold }
\usepackage{ tensor }
\usepackage{url}

\usepackage{bbm}

\usepackage[dvipsnames]{xcolor}
\usepackage{graphicx}
\usepackage{extpfeil}

\DeclareMathOperator{\Gra}{{Gr}}

\DeclareMathOperator{\pt}{{pt}}

\DeclareMathOperator{\ind}{{ind}}

\DeclareMathOperator{\reg}{{reg}}

\DeclareMathOperator{\Div}{{Div}}

\usepackage{tikz-cd}

\newcommand{\bA}{{\mathbb A}}

\newcommand{\bC}{{\mathbb C}}
\newcommand{\bD}{{\mathbb D}}

\newcommand{\bG}{{\mathbb G}}

\newcommand{\bL}{{\mathbb L}}

\newcommand{\bR}{{\mathbb R}}

\newcommand{\bZ}{{\mathbb Z}}

\newcommand{\cA}{{\mathcal A}}

\newcommand{\cD}{{\mathcal D}}
\newcommand{\cE}{{\mathcal E}}
\newcommand{\cF}{{\mathcal F}}
\newcommand{\cG}{{\mathcal G}}

\newcommand{\cI}{{\mathcal I}}

\newcommand{\cK}{{\mathcal K}}
\newcommand{\cL}{{\mathcal L}}

\newcommand{\cN}{{\mathcal N}}
\newcommand{\cO}{{\mathcal O}}
\newcommand{\cP}{{\mathcal P}}

\newcommand{\cT}{{\mathcal T}}

\newcommand{\cV}{{\mathcal V}}

\newcommand{\cY}{{\mathcal Y}}
\newcommand{\cZ}{{\mathcal Z}}

\newcommand{\Ranp}{{\on{Ran}}}

\newcommand{\eff}{\text{eff}}

\DeclareMathOperator*{\colim}{colim}

\newcommand{\nc}{\newcommand}
\nc{\renc}{\renewcommand}
\nc{\ssec}{\subsection}
\nc{\sssec}{\subsubsection}
\nc{\on}{\operatorname}

\nc\Gr{\on{Gr}}

\nc\Fl{\on{Fl}}

\newtheorem{thm}[subsubsection]{Theorem}

\newtheorem{conj}{Conjecture}
\DeclareMathOperator{\DGCat}{{DGCat}}

\DeclareMathOperator{\QCoh}{{QCoh}}

\DeclareMathOperator{\Maps}{{Maps}}

\DeclareMathOperator{\oblv}{{oblv}}

 \DeclareMathOperator{\Bun}{{Bun}}
\DeclareMathOperator{\codim}{{codim}}
\DeclareMathOperator{\gr}{{gr}}
\DeclareMathOperator{\irr}{{irr}}

\DeclareMathOperator{\Lie}{{Lie}}
\DeclareMathOperator{\LocSys}{{LS}}
\DeclareMathOperator{\Loc}{{Loc}}

\DeclareMathOperator{\Sym}{{Sym}}

\DeclareMathOperator{\bfmod}{{-\mathbf{mod}}}

\newcommand{\limto}{{\displaystyle\lim_{\longrightarrow}}}
\newcommand{\rightlim}{\mathop{\limto}}

\newcommand{\leftlim}{\mathop{\displaystyle\lim_{\longleftarrow}}}
\newcommand{\limfromn}{\leftlim\limits_{\raise3pt\hbox{$n$}}}
\newcommand{\limton}{\rightlim\limits_{\raise3pt\hbox{$n$}}}

\newcommand{\rightlimit}[1]{\mathop{\lim\limits_{\longrightarrow}}\limits%
	_{\raise3pt\hbox{$\scriptstyle #1$}}}

\newcommand{\leftlimit}[1]{\mathop{\lim\limits_{\longleftarrow}}\limits%
	_{\raise3pt\hbox{$\scriptstyle #1$}}}

\DeclareMathOperator{\Id}{{Id}}
\DeclareMathOperator{\Aut}{{Aut}}

 \DeclareMathOperator{\Ind}{{Ind}}

\DeclareMathOperator{\Ker}{{Ker}} \DeclareMathOperator{\id}{{id}}
 \DeclareMathOperator{\Mmod}{{-mod}}

\DeclareMathOperator{\Nilp}{{Nilp}} \DeclareMathOperator{\op}{{op}}

\DeclareMathOperator{\Vect}{{Vect}}

\DeclareMathOperator{\PGL}{{PGL}}
\DeclareMathOperator{\GL}{{GL}}

\DeclareMathOperator{\Irr}{{Irr}}
\DeclareMathOperator{\Ad}{{Ad}}
\DeclareMathOperator{\Rep}{{Rep}}

\makeatletter

\newcommand{\Rmnum}[1]{\expandafter\@slowromancap\romannumeral #1@}
\makeatother

\newtheorem{pr}[subsubsection]{Proposition}
\newtheorem{lm}[subsubsection]{Lemma}

\newtheorem{cor}[subsubsection]{Corollary}

\newtheorem{cnstr}[subsubsection]{Construction}

\newtheorem{df}[subsubsection]{Definition}

\newtheorem{rem}[subsubsection]{Remark}

\newtheorem{ex}[subsubsection]{Example}
\newtheorem{ntn}[subsubsection]{Notation}

\numberwithin{equation}{section}

\newcommand{\dR}{\mathrm{dR}}

\DeclareMathOperator{\D}{{DMod}}

\nc{\WFactAlg}{\on{WFactAlg}}
\nc{\FactAlg}{\on{FactAlg}}
\nc{\bFact}{\mathbf{Fact}}
\nc{\FactAlgCat}{\mathbf{FactAlgCat}}
\nc{\WFactAlgCat}{\mathbf{WFactAlgCat}}
\nc{\LWFactAlgCat}{\mathbf{LWFactAlgCat}}
\nc{\WFactModCat}{\on{-}\!\mathbf{WFactModCat}}
\nc{\LWFactModCat}{\on{-}\!\mathbf{LWFactModCat}}
\nc{\FactModCat}{\on{-}\!\mathbf{FactModCat}}
\nc{\WFactMod}{\on{-WFactMod}}
\nc{\FactMod}{\on{-FactMod}}
\nc{\weak}{{\on{weak}}}
\nc{\lax}{{\on{lax}}}
\nc{\str}{{\on{str}}}
\nc{\laxun}{{\on{laxun}}}
\nc{\strun}{{\on{strun}}}
\nc{\enh}{{\on{enh}}}

\nc{\an}{{\on{an}}}
\nc{\re}{{\on{Re}}}
\nc{\Shv}{{\on{Shv}}}
\nc{\Tot}{{\on{Tot}}}
\nc{\eq}{{\on{equiv}}}
\nc{\rh}{{\on{RH}}}
\nc{\rs}{{\on{r.s.}}}
\nc{\hol}{{\on{hol}}}
\nc{\Kos}{{\on{Kos}}}
\nc{\SL}{{\on{SL}}}
\nc{\redu}{{\on{red}}}
\nc{\CC}{{\on{CC}}}
\nc{\Sph}{{\on{Sph}}}
\nc{\Av}{{\on{Av}}}
\nc{\temp}{{\on{Wh-temp}}}
\nc{\atemp}{{\on{Wh-anti-temp}}}
\nc{\Whit}{{\on{Whit}}}
\nc{\irreg}{{\on{irreg}}}
\nc{\coeff}{{\on{coeff}}}
\nc{\loc}{{\on{loc}}}
\nc{\rel}{{\on{rel}}}
\nc{\lvl}{{\on{lvl}}}
\nc{\Sat}{{\on{Sat}}}
\nc{\SingS}{{\on{SS}}}

\begin{document}
	
	\title[Non-vanishing of quantum geometric Whittaker coefficients]{Non-vanishing of quantum geometric Whittaker coefficients}
	
	\author[E.~Bogdanova]{Ekaterina Bogdanova}
	\address{Harvard University,  USA}
	\email{ebogdanova@math.harvard.edu}
	
	\begin{abstract}
	We prove that for any reductive group $G$ of adjoint type cuspidal automorphic twisted D-modules have non-vanishing quantum Whittaker coefficients. The argument provides a  microlocal interpretation of quantum Whittaker coefficients for any $\check{\Lambda}^+$-valued divisor under some hypothesis on singular support.
	\end{abstract}
	
	\maketitle
	
	\tableofcontents
	
	\section{Introduction.}
	
		\subsubsection{} Recall that any cuspidal (i.e. with zero constant term) modular form has at least one non-zero Fourier coefficient. The statement also holds for automorphic forms for $\GL_n$ and Whittaker coefficients. For $\GL_2$, it is also true for {\it metaplectic} automorphic forms (\cite{GHPS}), which are natural generalizations to arbitrary groups of modular forms with fractional weights. However, it fails for arbitrary $G$ (\cite{HPS}). 
	
	\subsubsection{} Miraculously, the geometric analog of the question, non-vanishing  of Whittaker coefficients for automorphic sheaves, holds. This was proved recently by Faergeman-Raskin (\cite{FR}), and turned out to be an important step in the proof of the global unramified geometric Langlands conjecture. In the present paper, we study this question in the context of the {\it quantum Geometric Langlands program}. Namely, we prove that the non-vanishing  of Whittaker coefficients for {\it metaplectic} automorphic sheaves holds as well. 

	\subsection{Metaplectic and quantum Geometric Langlands theory.}
	\subsubsection{}  We will work over the field $k$ of complex numbers. Fix a smooth proper curve $X$ over an algebraically closed field of characteristic zero. One can observe that the main objects of study in the global geometric Langlands theory (the category of D-modules on the stack of $G$-bundles $\Bun_G$, the category of D-modules on the affine Grassmannian $\Gra_G$, etc) deform over the space of \emph{Kac-Moody levels} $\kappa$ of $G$, suggesting the existence of \emph{quantum} Langlands correspondence, deforming the usual one. For a fixed level $\kappa$ denote the corresponding category by $\D_{\kappa}(\Bun_G)$.

	This is the geometric counterpart of the space of metaplectic automorphic functions. Moreover, Kac-Moody levels admit duality as well: $\kappa$ gives a level $\check{\kappa}$ for the Langlands dual group $\check{G}$. Then the quantum geometric Langlands theory expects (roughly) the following 
	\begin{conj}[Gaitsgory-Lysenko, Drinfeld, Stoyanovsky...]\label{qGLC}
		\begin{equation}
			\D_{\kappa}(\Bun_G) \xrightarrow[\bL_{\kappa}]{\sim} \D_{-\check{\kappa}}(\Bun_{\check{G}}).
		\end{equation}
	\end{conj}

	\begin{rem}
		When $G$ is simple, a choice of $\kappa$ is just a choice of an element $c$ in the base field such that $\kappa = c \kappa_{\operatorname{b}}$ (here $\kappa_{\operatorname{b}}$ is the basic form, i.e. the one such that the short coroots have squared length two). Then the $\D_c(\Bun_G):=\D_{\kappa}(\Bun_G)$ is the category of D-modules on $\Bun_G$ twisted by the $(\frac{c - h^{\vee}}{2h^\vee})$-th power of the determinant line bundle, where $h^\vee$ is the dual Coxeter number of $G$. In this case, conjecture \ref{qGLC} reads as 
			\begin{equation}
		\D_{c}(\Bun_G) \xrightarrow[\bL_{\kappa}]{\sim} \D_{-\frac{1}{rc}}(\Bun_{\check{G}}),
		\end{equation}
		where $r=1$, 2 or 3 is the maximal multiplicity of arrows in the Dynkin diagram of $G$.
	\end{rem}

\begin{rem}
	In the limit $c \mapsto 0$, the category $\D_{-\frac{1}{rc}}(\Bun_{\check{G}})$ becomes $\QCoh(\LocSys_{\check{G}})$ (\cite[Section 6]{Z}). So Conjecture \ref{qGLC} limits to the global Geometric Langlands equivalence, conjectured by Beilinson and Drinfeld and established in \cite{GLCI}, \cite{GLCII}, \cite{GLCIII}, \cite{GLCIV},  and \cite{GLCV}. This equivalence (roughly) says that 
	\begin{equation}
	\D_{\kappa_{\operatorname{crit}}}(\Bun_G) \cong \QCoh(\LocSys_{\check{G}}).\footnote{In classical Geometric Langlands, the precise formulation involves a certain renormalization: one considers the category $\operatorname{IndCoh}_{\Nilp}(\LocSys_{\check{G}})$ instead. However, it is expected that in the quantum version one does not need this procedure.}
	\end{equation}
\end{rem}

	\subsubsection{}  We would also like to relate the automorphic side to a space of certain local systems/L-parameters, and this is where the geometric metaplectic Langlands theory comes in. 
	
	We will need the following construction from introduced in \cite[Section 6]{GL}. Using $\kappa$ we modify the root datum of $G$ to obtain a new reductive group $H$. Moreover, from the data of $\kappa$ we also get a gerbe $\cG_{Z_H}$ with respect to the center $Z_H$ of $H$, and a character $\epsilon: \pm \rightarrow Z_H$. This triple is called \emph{the metaplectic dual root datum}. 
	
	Then the $H$-local systems on $X$ twisted by $\cG_{Z_H}$ provide the desired space of spectral parameters (\cite[9.6.6]{GL}). In particular, we have the Hecke action of the factorization category $\Rep(H)$ on $\D_{\kappa}(\Bun_G)$, which conjecturally factors through the category of quasi-coherent sheaves on the space $\LocSys_H^{\cG_Z})$ of $\cG_{Z_H}$-twisted $H$-local systems on $X$.
	
	\subsubsection{} However, both arithmetic metaplectic theory and geometric quantum theory have a major defect: the Hecke action is more degenerate. For example, while in the usual geometric story (i.e. when $\kappa$ is the critical level) for an irreducible $\check{G}$-local system $\sigma$ there exists a unique Hecke eigensheaf $\cA_{\sigma}$ with eigenvalue $\sigma$, in the quantum setting for irrational $\kappa$ the group $H$ is trivial, and therefore the Hecke eigensheaf condition is trivial as well. But there is a tool that provides finer information than Hecke action: Whittaker coefficients, which is the main object of study in the present paper. 
	\subsection{Main result.}
	
	\subsubsection{} Our first result is concerned with a certain subcategory of $\D_{\kappa}(\Bun_G)$. This is the category $\Shv_{\kappa, \Nilp}(\Bun_G)$, where all the objects are (colimits of) twisted holonomic D-modules with regular singularities with singular support in $\Nilp \subset T^*Bun_G$, which is the zero fiber of the Hitchin fibration. Namely, we prove 
	\begin{thm}\label{whitshvnilp}
		Let $G$ be an adjoint  reductive group, $\kappa$ a rational level of $G$. Then for any cuspidal $\cF \in \Shv_{\kappa, \Nilp}(\Bun_G)$ there exists a $\check{\Lambda}^+$-valued divisor $D$ such that $\coeff_D(\cF) \neq 0$.
	\end{thm}
Here $\coeff_D$ is the Whittaker coefficient functor. We give an overview and definitions in Section \ref{coeffs}.
	
	\begin{rem}
		The statement of Theorem \ref{whitshvnilp} is parallel to the results of \cite{FR}. However, the proof given in \cite{FR} heavily relies on the Hecke action, and therefore does not work in the quantum context.
	\end{rem}

\subsubsection{} As in \cite{FR}, we in fact prove a stronger assertion than Theorem \ref{whitshvnilp}. Namely, we prove the statement for \emph{Whittaker-tempered} twisted D-modules. 

\begin{rem}
	We show (Section \ref{nice}) that Whittaker-temperedness condition is in fact vacuous for ``most"  rational levels. This is another effect in the quantum theory which does not apper at the critical level.
\end{rem}

\subsubsection{} The theory of \cite[Section 20]{AGKRRV} reduces understanding the cateogry  $D_{\kappa}(\Bun_G)$ to studying $\Shv_{\kappa, \Nilp}(\Bun_G)$, provided there exists a natural action of $\QCoh(\LocSys_H^{\cG_Z}))$ on $D_{\kappa}(\Bun_G)$. This result is not available yet, but was conjectures by Gaitsgory and Lysenko in \cite{GL}. We formulate the conjecture precisely in  Section \ref{SA}. In the rest of the subsection, we give a summary of our other results which assume this statement. 

\subsubsection{} The collection of the functors $\{\coeff_D \}$ for all $\check{\Lambda}^+$-valued divisors can be assembled in one functor  $$\D(\Bun_G) \rightarrow \D_{\kappa}(\Gra_{G, \operatorname{Conf}})^{LN, \chi},$$
where $\operatorname{Conf}$ is the configuration space of points of $X$ and $$\chi: LN \rightarrow \bA^1$$ is the non-degenerate character of $LN$. For a precise definition see Section \ref{coeffs}.

We deduce from Theorem \ref{whitshvnilp} the following result:
\begin{thm}\label{generalintro}
	The functor 
	$$\coeff^{\loc}: D_{\kappa}(\Bun_{G})^{\temp} \rightarrow \Whit_{\kappa}(\Gra_G)$$ is conservative.
\end{thm}

\begin{rem}
	The stronger version of the statement of Theorems \ref{whitshvnilp} and \ref{generalintro}, without assuming that $G$ is adjoint, will be established as a corollary in a forthcoming work \cite{B}.
\end{rem}

	\subsection{Outline of the argument.}
In this subsection, we provide a summary of the proof of Theorem \ref{whitshvnilp}. There are additional sheaf-theoretic tools we can use to study $\Shv_{\Nilp}(\Bun_G)$: microlocal geometry (\cite{KS}).

\subsubsection{} We first discuss some geometry of the stack $\Nilp$ (see section \ref{higgs} for more details). Let $$\Nilp^{\reg} \subset \Nilp$$ denote the open substack of {\it regular} nilpotent Higgs fields. This is a global version of the space of regular nilpotent elements in $\mathfrak{g}$. The stack $\Nilp^{\reg}$ is the union of equidimensional smooth irreducible components numbered by a sublattice  $\check{\Lambda}^{\operatorname{rel}} \subset \check{\Lambda}$ .

Following \cite{FR} we prove that if $\cF$ is Whittaker-tempered then $\SingS(\cF) \cap \Nilp^{\reg} \neq \emptyset$. 

\subsubsection{} We now discuss microlocal geometry behind the functors $\coeff_D$. We emphasize that this approach is entirely different from the one in \cite{FR}. We define this functor as 
$$\cF \mapsto C_{\dR}(\Bun_N^{\omega(-D)} , p^!(\cF) \otimes \psi_D^!(\operatorname{exp}))[-\dim \Bun_N^{\omega(-D)}],$$
where $p: \Bun_N^{\omega(-D)} \rightarrow \Bun_G$ is the form of $\Bun_N$ corresponding to $D$, and $\psi_D$ is the Whittaker function on $\Bun_N^{\omega(-D)}$ (see Section \ref{coeffs} for more details). 
	
The Whittaker function defines a Lagrangian $$\Bun_N^{\omega(-D)} \rightarrow T^*\Bun_N^{\omega(-D)}.$$ Compose this Lagrangian with the natural Lagrangian correspondence between $T^*\Bun_N^{\omega(-D)}$ and $T^*\Bun_G$. We denote the resulting Lagrangian in $T^*\Bun_G$ by $\Kos_D$ and call it the {\it $D$-Kostant slice} (since for $D=0$ this stack is the global Kostant slice). 

The following result is proved (up to some modifications discussed in Section \ref{NT}) in \cite{NT}. The stack $\Kos_D$ controls the functor $\coeff_D$ in the following sense: for a closed conic Lagrangian $\Lambda \subset T^* \Bun_G$
\begin{thm}\label{thm intro}
	If the shifted conormal bundle $\Kos_D$ intersects $\Lambda$ transversely at a smooth point $\{\lambda_D\}$, then $\coeff_D|_{\ \Shv_{\kappa, \Lambda}(\Bun_G)}$ calculates the twisted microstalk at $\{\lambda_D\}$.
\end{thm}
In particular, this means that $\coeff_D|_{\ \Shv_{\kappa, \Lambda}(\Bun_G)}$ is t-exact and commutes with Verdier duality. For a constructible $\cF$ denote by $$\operatorname{CC}(\cF) = \sum_{\beta \in \Irr(\Lambda)}^{} c_{\beta, \cF} \cdot [\beta]$$
its characteristic cycle. Denote by $\beta_D$ the irreducible component such that $\lambda_D \in \beta_D$.
\begin{cor}\label{cor intro}
	For $\cF \in \Shv_{\kappa, \Lambda}(\Bun_G)^{\operatorname{cnstr}}$, 
	we have $$\chi(\coeff_D(\cF)) = c_{\beta_D, \cF}.$$
\end{cor}

\subsubsection{} It is well-known that for $D=0$ the stack $\Kos_0$ intersects $\Nilp$ transversely at a single smooth point. Although this is no longer true for other $D$, we find that $\Kos_D$ intersects $\underline{\Nilp}_{\redu}^{\reg, \deg(D)}$ (i.e. the union of irreducible component of $\Nilp_{\redu}^{\reg}$ with of number not less than $\deg(D)$) transversely at a smooth point. We do so by embedding the intersection $\Kos_D \cap \Nilp_{\redu}^{\reg}$ into the Zastava space and matching connected components of the latter with intersections of $\Kos_D$ with different irreducible component of $\Nilp_{\redu}^{\reg}$.

Moreover, it is easy to see that $\Kos_D \cap\Nilp_{\redu}^{\irreg} = \emptyset$. Therefore the conditions of Theorem \ref{thm intro} are satisfied for $$\Lambda = \Nilp_{\redu} \setminus (\Nilp_{\redu}^{\reg})^{< \deg(D)}.$$

\subsubsection{} Finally, given $\cF \in \Shv_{\kappa, \Nilp}(\Bun_G)^{\temp}$ choose a minimal $\check{\lambda} \in \check{\Lambda}^{\rel}$ such that $$\SingS(\cF) \cap \Nilp^{\reg, \check{\lambda}} \neq 0$$ and a $\check{\Lambda}$-valued divisor $D$ on $X$ with $$\check{\lambda} = (2-2g)\check{\rho} + \deg(D).$$ We claim that $\coeff_D(\cF) \neq 0$.   

Indeed, there exists $i$ and a constructible $\cG \in H^i(\cF)$ with $\Nilp^{\reg, \check{\lambda}}  \in \SingS(\cG)$. Then we by Corollary \ref{cor intro} have $\coeff_D(\cG) \neq 0$. By since the functor is t-exact we have 
$$H^0(\coeff_D(\cG)) \hookrightarrow H^0(\coeff_D(H^i(\cF))) \cong H^i(\coeff_D(\cF)),$$
which implies the claim.

	\subsection{Notations and conventions.}
	
		\subsubsection{Categories.} We use the formalism of $\infty$-categories developed in \cite{LHTT}, \cite{HA}, and theory of DG categories as understood in \cite{GR}. By a DG category we mean $k$-linear stable $\infty$-category. We let $\DGCat$ be the category
	of presentable DG categories with continuous functors with monoidal structure given by the Lurie tensor product $\otimes$. 
	
	\subsubsection{Lie theory.}
	
	Throughout the paper $G$ will stand for a reductive group over $k$ of adjoint type. We choose Borel subgroup $B \subset G$, the opposite Borel $B^- \subset G$ and the Cartan subgroup $T= B \cap B^-$. Let $N$ ($N^-$) be the unipotent radical of $B$ ($B^-$). Let $\Lambda$ ($\check{\Lambda}$) denote the lattice of weights (coweights) of $G$. Let $\Lambda^+$ ($\check{\Lambda}^+$) denote the subset of dominant weights (coweights). Let $\cI$ denote the set of nodes for the Dynkin diagram of $G$. For $i \in \cI$, we let $\alpha_i$ ($\check{\alpha}_i$) denote the
corresponding simple root (coroot). Let $2\rho$ denote the sum of simple roots and $2\check{\rho}$ the sum of simple coroots. Let $\check{G}$ be the Langlands dual group over $k$ for $G$. We denote by $LG$ ($L^+G$) the loop (arc) group of $G$.

Let $\mathfrak{g}_{\irr} \subset \mathfrak{g}$ denote the reduced closed subscheme of irregular elements.
We let $\cN \subset \mathfrak{g}$ denote the nilpotent cone. We let $\cN^{\irr} := \cN \cap \mathfrak{g}_{\irr}$ denote the subscheme of irregular
nilpotent elements. We let $\cN^{\reg} \subset \cN$ denote the open complement to $\cN^{\irr}$, which parametrizes
of regular nilpotent elements.
	
\subsubsection{Levels.} A level $\kappa$ for $G$ is a $G$-invariant symmetric bilinear form 
$$\kappa: \Sym^2(\mathfrak{g}) \rightarrow \bC.$$
Denote by $\kappa_{\mathfrak{g}, \operatorname{crit}}$ the critical level, i.e. $-\frac{1}{2}$ times the Killing form. 
Throughout the paper we will assume that the level $\kappa$ is non-degenerate, i.e. that $\kappa - \kappa_{\mathfrak{g}, \operatorname{crit}}$ is nondegenerate as a bilinear form. The {\it dual} level $\check{\kappa}$ on $\check{G}$ for $\kappa$ is the unique nondegenerate level such that the restriction of $\check{\kappa} - \kappa_{\check{\mathfrak{g}}, \operatorname{crit}}$ to $\mathfrak{t}^*$ and the restriction of $\kappa - \kappa_{\mathfrak{g}, \operatorname{crit}}$ to $\mathfrak{t}$ are dual symmetric bilinear forms.

Suppose $G$ is simple. A level $\kappa$ is {\it rational} if $\kappa$ is a rational multiple of the Killing form and {\it irrational}otherwise. A level $\kappa$ is {\it positive} if $\kappa - \kappa_{\mathfrak{g}, \operatorname{crit}}$ is a positive rational multiple of the Killing form. A level is {\it negative} if it is not positive or critical. For general reductive $G$, we say a level $\kappa$ is rational, irrational, positive, or negative if its restrictions to each simple factor are so.
	
	\subsubsection{(Twisted) D-modules.} 
	For a prestack $\cY$ over $k$ equipped with a twisting  (in the sense of \cite{GRCrys}) $\cT$ denote by $\D_{\cT}(\cY)$ the category of twisted crystals defined and studies in \cite{GRCrys}. For a map $f: \cY \rightarrow \cZ$ of stacks with twistings $\cT_{\cY}$ and $\cT_{\cZ}$ such that $f^!\cT_{\cZ} \cong \cT_{\cY}$, we let $$f^!: \D_{\cT_{\cZ}}(\cZ) \rightarrow \D_{\cT_{\cY}}(\cY)$$ denote the
corresponding pullback functor. If $f$ is ind-representable, we let $$f_*: \D_{\cT_{\cY}}(\cY) \rightarrow \D_{\cT_{\cZ}}(\cZ)$$ the
corresponding pushforward functor. When defined, we denote by $f_!$ and $f^*$ the left adjoint functors.

	Recall that a level $\kappa$ gives rise to a (factorizable) de Rham twisting $\cG_{\kappa}$ on the affine Grassmannian $\Gra_G$, which descends to $\Bun_{G}$ (e.g. \cite[10.1]{GLCII}). We set the notation 
	$$\D_{\kappa}(\Bun_G) := \D_{\cG_{\kappa}}(\Bun_G).$$
	
	\subsubsection{(Twisted) sheaves.} 
	For an affine scheme $S$ of finite type equipped with a twisting $\cT$, we let $$\Shv_{\cT}(S)^c \subseteq \D_{\cT}(S)$$
be the subcategory of compact objects that are holonomic with regular singularities. Let $$\Shv_{\cT}(S) = \Ind(\Shv_{\cT}(S)^c) \subset \D_{\cT}(S).$$
For $Y$ an algebraic stack equipped with a twisting $\cT$, we let 
$$\Shv_{\cT}(Y):= \lim_{s: S \rightarrow Y}\Shv_{\cT|_S}(S),$$
and $$\Shv_{\cT}(Y)^{\operatorname{cnstr}}:= \lim_{s: S \rightarrow Y}\Shv_{\cT|_S}(S)^c.$$
One can show that $\Shv_{\cT}(Y)$ has a natural t-structure, and $\Shv_{\cT}(Y)^{\operatorname{cnstr}}$ is closed under truncations. We refer to objects in $\Shv_{\cT}(Y)$ as twisted sheaves, and to objects in $\Shv_{\cT}(Y)^{\operatorname{cnstr}}$ as twisted constructible sheaves. On $\Shv_{\cT}(Y)^{\operatorname{cnstr}}$ we have a well-defined Verdier duality equivalence, denoted by $$\bD: \Shv_{\cT}(Y)^{\operatorname{cnstr}, \op} \cong \Shv_{\cT}(Y)^{\operatorname{cnstr}}.$$
For $\Lambda \subset T^*Y$ a closed conical substack we let $$\D_{\cT, \Lambda}(Y) \subseteq \D_{\cT}(Y),$$
$$\Shv_{\cT, \Lambda}(Y) \subseteq \Shv_{\cT}(Y)$$
denote the subcategory of twisted D-modules/sheaves with singular support in $\Lambda$.

	\subsection{Acknowledgments.} This work owes its existence to Sam Raskin. I am grateful to him for proposing this project, and for being generous with his time and mathematics. I thank Dennis Gaitsgory, Kevin Lin, Lin Chen, Gurbir Dhillon, Joakim Faergeman, Wyatt Reeves, and Sasha Petrov for helpful discussions. I am also grateful to Dennis Gaitsgory, Kevin Lin, and Sam Raskin for their comments on an earlier version of this text.

	\section{Generalities.}
	
		\subsection{Higgs bundles.}\label{higgs}
	Recall that a Higgs bundle on $X$ for the group $G$ is a pair $(\cP_G, \phi)$ where $\cP_G$ is a
	$G$-bundle and $\phi \in \Gamma(X, \mathfrak{g}_{\cP_G} \otimes \omega) $.
	Recall that Higgs bundles form an algebraic stack, which can be written as a mapping
	stack
	$$\Maps(X, \mathfrak{g}/(G \times\bG_m)) \times_{\Maps(X,\bG_m)} \{\omega\}.$$
	Choosing an identification $\mathfrak{g} \cong \mathfrak{g}^*$ we get an isomorphism between the stack of Higgs bundles and $T^*\Bun_G$.
	
	\begin{ntn}
		Let $\cY$ be a stack and $\stackrel{\circ}{\cY} \subset \cY$ be an open substack. Let $$\Maps(X, \stackrel{\circ}{\cY} \subset \cY)$$ denote the stack with $S$-points given by maps 
		$$y: X \times S \rightarrow \cY$$
		such that there exists an open $U \subseteq X \times S$ such that:
		\begin{itemize}
			\item $U \subseteq X\times S$ is schematically dense;
			\item $U \rightarrow S$ is a flat cover;
			\item $y|_{U}$ factors through $\stackrel{\circ}{\cY}$.
		\end{itemize}
	\end{ntn}
	
	Recall the definition of the \emph{global nilpotent cone}: $$\Nilp:= \Maps(X, \cN / (G \times \bG_m))\times_{\Maps(X,\bG_m)} \{\omega\} \subset T^* \Bun_G.$$ Recall the stack of \emph{irregular nilpotent Higgs bundles}
	$$\Nilp^{\irr} := \Nilp:= \Maps(X, \cN^{\irr} / (G \times \bG_m))\times_{\Maps(X,\bG_m)} \{\omega\} \subset T^* \Bun_G$$
	and its open complement $$\Nilp^{\reg} := \Maps(X, \cN^{\reg} / (G \times \bG_m) \subseteq \cN/ (G \times \bG_m))\times_{\Maps(X,\bG_m)} \{\omega\},$$
	the stack of \emph{generically regular Higgs bundles}.
	
	\subsubsection{Generically regular nilpotent Higgs bundles.} Next we record some geometric properties of $\Nilp^{\reg}$ established in \cite{BD} and \cite[2.5]{FR}, which will be crucial for our arguments. Let $\stackrel{\circ}{\mathfrak{n}} \subset \mathfrak{n}$ denote the open subscheme of elements of $\mathfrak{n}$ that are regular as elements of $\mathfrak{g}$. Set 
	$$\widetilde{\Nilp}^{\reg} := \Maps(X, \stackrel{\circ}{\mathfrak{n}}/(B \times \bG_m) \subseteq \mathfrak{n}/(B \times \bG_m)) \times_{\Maps(X,\bG_m)} \{\omega\}.$$
	We have canonical maps 
	$$\Nilp^{\reg} \leftarrow \widetilde{\Nilp}^{\reg} \rightarrow \Bun_B \rightarrow \Bun_T.$$
	\begin{ntn}
		For a coweight $\check{\varphi}$ let $$\widetilde{\Nilp}^{\reg, \check{\varphi}}:= \widetilde{\Nilp}^{\reg} \times_{\Bun_T} \Bun_T^{\check{\varphi}}.$$ 
	\end{ntn}
	Note that $$\widetilde{\Nilp}^{\reg, \check{\varphi}} \rightarrow \Nilp^{\reg}$$ is a locally closed embedding, and the corresponding locally closed subschemes define a stratification of $\Nilp^{\reg}$. In more detail, we have 
	\begin{pr}\cite[Prposition 2.5.4.1]{FR}\label{2.5.4.1} For a coweight $\check{\varphi}$
		\begin{enumerate}
			\item $\widetilde{\Nilp}^{\reg, \check{\varphi}}$ is smooth of dimension $\dim \Bun_G$. 
			\item Suppose $(\alpha_i,\check{\varphi}) + 2g - 2 < 0$ for some $i \in \cI$. Then $\widetilde{\Nilp}^{\reg, \check{\varphi}}$ is empty.
			\item Suppose $(\alpha_i,\check{\varphi})+ 2g - 2 \geq 0$ for all $i \in \cI$. Then $\widetilde{\Nilp}^{\reg, \check{\varphi}}$ is non-empty and connected. \footnote{Importantly, this is only true when the center of $G$ is trivial. This is the only place where we use that $G$ is adjoint.}
		\end{enumerate}
	\end{pr}
	Denote by $\check{\Lambda}^{\operatorname{rel}} \subset \check{\Lambda}$ the subset of coweights $\check{\phi}$ such that $(\alpha_i,\check{\varphi})+ 2g - 2 \geq 0$ for all $i \in \cI$.

	\subsection{Coefficient functors.}\label{coeffs}
	Fix $\omega^{\frac{1}{2}}$ a square root of the canonical sheaf on $X$. Set $\check{\rho}(\omega):= (2\check{\rho})(\omega^{\frac{1}{2}})$. Let $D$ be a $\check{\Lambda}$-valued divisor on $X$. let $\cO(D)$ be a $T$-bundle on $X$ characterized by the property that for every $\alpha_i: T \rightarrow \bG_m$ we have that the induced bundle $\alpha_i(\cO(D))= \cO(\alpha_i(D))$. Let $\omega(D)$ denote the product of $\check{\rho}(\omega)$ and $\cO(D)$. Set $$\Bun_N^{\omega(D)}:= \Bun_B \times_{\Bun_T} \{\omega(D)\}.$$
	Recall from \cite[4.1.3]{FGV} that for $D$ a $\check{\Lambda}^+$-valued divisor we have a canonical function
	$$\psi_D: \Bun_N^{\omega(-D)} \rightarrow \bA^1.$$
	\begin{df}
		For such  a $\check{\Lambda}^+$-valued divisor $D$ on $X$ define $$\coeff_D: \D_{\kappa}(\Bun_G) \rightarrow \Vect$$
		by the formula
		$$\cF \mapsto C_{\dR}(\Bun_N^{\omega(-D)} , p^!(\cF) \otimes \psi_D^!(\operatorname{exp}))[-\dim \Bun_N^{\omega(-D)}],$$
		where $p: \Bun_N^{\omega(-D)} \rightarrow \Bun_G$.
	\end{df}
	Note that taking de Rham cohomology makes sense since the twisting $\cG_{\kappa}$ is canonically trivial on $\Bun_N^{\omega(-D)}$. Here $\operatorname{exp}$ stands for the exponential sheaf on $\bA^1$.
	
	\subsubsection{} The collection of the functors $\{\coeff_D \}$ for all $\check{\Lambda}^+$-valued divisors can be assembled in one functor $\coeff^{\operatorname{loc}}$ as follows (\cite[9.3]{GLCII}). Let $\Gra_{G, \omega}$ be the $\check{\rho}(\omega)$-twist of the Beilinson-Drinfeld grassmannian in the sense of \cite[1.2]{GLCII}. Consider the map 
	$$
	\pi: \Gra_{G, \omega} \rightarrow \Bun_G,
	$$
	and the corresponding functor
	$$\pi^!: \D_{\kappa}(\Bun_G) \rightarrow \D_{\kappa}(\Gra_{G, \omega})^{LN_{\omega}, \chi},$$
	where $\chi$ is the non-degenerate character. 
	To simplify the notation, we will work over a fixed point $\underline{x} \in \Ranp$ and will define a functor 
	$$\coeff^{\operatorname{loc}}_{\underline{x}}: \D(\Bun_G) \rightarrow \D_{\kappa}(\Gra_{G, \omega, \underline{x}})^{LN_{\omega}, \chi}.$$
	Note that $LN_{\omega}$ is a filtered union of subschemes $N^{\alpha}$. For each $\alpha$ consider
	$$\Av_*^{N^{\alpha}, \chi}: \D_{\kappa}(\Gra_{G, \omega, \underline{x}}) \rightarrow \D_{\kappa}(\Gra_{G, \omega, \underline{x}})^{N^{\alpha}, \chi} \hookrightarrow \D_{\kappa}(\Gra_{G, \omega, \underline{x}}) .$$
	For $N^{\alpha} \subset N^{\alpha^{\prime}}$ we have a natural transformation
	\begin{equation}\label{nattr}
	\Av_*^{N^{\alpha^{\prime}}} \rightarrow \Av_*^{N^{\alpha}}.
	\end{equation}
	\begin{lm}{\cite[Lemma 9.3.4]{GLCII}}
		For $N^{\alpha}$ large enough the natural transformation 
		$$\Av_*^{N^{\alpha^{\prime}}} \circ \pi_{\underline{x}}^! \rightarrow \Av_*^{N^{\alpha}} \circ \pi_{\underline{x}}^!$$
		induced by (\ref{nattr}) is an equivalence.
	\end{lm}
	Therefore for large enough $N^{\alpha}$ the functors $ \Av_*^{N^{\alpha}}$ coincide. So in particular, they map to $$\cap_{N^{\alpha}}\D_{\kappa}(\Gra_{G, \omega, \underline{x}})^{N^{\alpha}, \chi} \cong \D_{\kappa}(\Gra_{G, \omega, \underline{x}})^{LN_{\omega}, \chi}.$$
	The resulting functor is $\coeff^{\operatorname{loc}}_{\underline{x}}$. For a fixed $\check{\Lambda}^+$-valued divisor $D$ consider the composition 
	\begin{equation}
	\coeff^{\operatorname{loc}}_D: \D_{\kappa}(\Bun_G) \xrightarrow{\coeff^{\operatorname{loc}}} \D_{\kappa}(\Gra_{G, \omega})^{N^{\alpha}, \chi} \xrightarrow{(-)^!} \D_{\kappa}(\Gra_{B, \omega, D})^{N^{\alpha}, \chi} \cong \Vect,
	\end{equation}
	where $\Gra_{B, \omega, D}$ is the connected component of $\Gra_B$ corresponding to $D$. 
	\begin{lm}
		We have $$\coeff_D\cong \coeff^{\operatorname{loc}}_D[\dim \Bun_N^{\omega(-D)}].$$
	\end{lm}
	
	\begin{proof}
		As in \cite[Lemma 9.6.7]{GLCII} we will compare the dual functors, denoted by $$\operatorname{Poinc}_{*, D}[\dim \Bun_N^{\omega(-D)}]$$ and $$\operatorname{Poinc}_{*, D}^{\operatorname{loc}}[\dim \Bun_N^{\omega(-D)}]$$ respectively. Explicitly, we have that $\operatorname{Poinc}_{*, D}$ as an object in $\D_{\kappa}(\Bun_G)_{\operatorname{co}}$ is given by 
		$$p_*(\psi_D)^*(\exp).$$
		Let us also describe $\operatorname{Poinc}_{*, D}^{\operatorname{loc}}$. Consider a map
		\begin{equation}\label{GR_B to Bun_G}
		\Gra_{B, \omega, D} \hookrightarrow \Gra_{G, \omega} \rightarrow \Bun_G.
		\end{equation}
		Note that the map (\ref{GR_B to Bun_G}) factors as $$\Gra_{B, \omega, D} \xrightarrow{\pi_{N, \omega, D}}\Bun_N^{\omega(-D)} \rightarrow \Bun_G.$$
		Then $\operatorname{Poinc}_{*, D}^{\operatorname{loc}}$ is given by the pushforward of $(\pi_{N, \omega, D} \circ \psi_D)^*(\exp)$ along (\ref{GR_B to Bun_G}). Hence it suffices to show that 
		\begin{equation}\label{9.19}
		(\pi_{N, \omega, D})_*(\pi_{N, \omega, D} \circ \psi_D)^*(\exp) \cong \psi_D^*(\exp).
		\end{equation}
		However, the action of $LT$ on $$\Gra_{B, \omega, D} \xrightarrow{\pi_{N, \omega, D}}\Bun_N^{\omega(-D)} $$
		identifies it with 
		$$\Gra_{N, \omega} \xrightarrow{\pi_{N, \omega}}\Bun_N^{\omega},$$
		and the character $\psi_D$ maps to $\psi_0$ under this identification. Thus, since \cite[(9.19)]{GLCII} holds, we have that (\ref{9.19}) also holds.
	\end{proof}

	\subsection{Singular support of twisted D-modules.}
	
	\subsubsection{Modules over TDOs and singular support.} Let $X$ be a smooth classical scheme, $\cD$ be a sheaf of TDOs. The goal of this subsection is to define singular support of coherent modules over $\cD$. The definition will follow the one for the usual D-modules. 
	
	\begin{df}
		Let $(M, F)$ be a filtered $\cD$-module. We say that $F$ is a good
		filtration of M if the following equivalent conditions are satisfied: 
		\begin{enumerate}
			\item $\gr^F M$ is coherent over $\pi_* \cO_{T^*X}$,
			\item $F_iM$ is coherent over $\cO_X$ for each $i$, and there exists $i_0 \gg 0$ such that for every $j \geq 0$, $i \geq i_0$
			$$(F_j\cD) F_iM = F_{i+j}M.$$
			\item There exist locally a surjective $\cD$-linear morphism $\Phi: \cD^{\oplus m} \rightarrow M$ and integers $n_j$, $j=1, ..., m$ such that for every $i \in \bZ$
			$$\Phi(F_{i-n_1}\cD \oplus F_{i-n_2}\cD \oplus ... \oplus F_{i-n_m}\cD) = F_iM.$$
		\end{enumerate}
	\end{df}
	
	\begin{pr}
		\begin{enumerate}
			\item A $\cD$-module is coherent if and only if it admits a globally defined good filtration.
			\item Let $F$, $F^{\prime}$ be two filtrations of a  $\cD$-module $M$ and assume that $F$ is good. Then
			there exists $i \gg 0$  such that for every $i \in \bZ$
			$$F_iM \subset F_{i+i_0}^{\prime}M.$$
			Moreover, if $F^{\prime}$ is also a good filtration there exists $i \gg 0$ such that for every $i \in \bZ$
			$$F_{i-i_0}^{\prime}M \subset F_iM \subset F_{i+i_0}^{\prime}M.$$
		\end{enumerate}
	\end{pr}
	
	\begin{proof}
		Follows \cite[Theorem 2.3]{HTT}.
	\end{proof}
	
	\begin{df}
		Let $M$ be a coherent $\cD$-module and choose a good filtration $F$ on it. We define the singular support $\SingS(M)$ of $M$ to be the support of the coherent $\cO_{T^*X}$-module $$\cO_{T^*X} \otimes_{\pi^* \pi_* \cO_{T^*X}} \pi^*(\gr^F M).$$
	\end{df}
	
	\begin{lm}
		The variety $\SingS(M)$ does not depend on the choice of a good filtration. 
	\end{lm}
	
	\begin{proof}
		Follows \cite[Theorem 2.2.1]{HTT}.
	\end{proof}
	
	\subsubsection{Singular support of twisted crystals.}
	
	Let $S$ be a smooth affine scheme with a twisting $\cT$ on it. Let $\cN \subset T^*S$ be a conical Zariski-closed subset. 
	
	\begin{df}
		We define the $\D_{\cT, \cN}(S)$ be ind-completion of the category $$\D_{\cT, \cN}(S)^{\text{f.g.}} \subset \D_{\cT}(S)^{\text{f.g.}},$$ obtained by requiring that each cohomology belong to $\D_{\cT, \cN}(S)^{\text{f.g.}, \heartsuit}$, where the latter is defined as above. 
	\end{df}
	
	\begin{rem}
		As in \cite[A.3.3-A.3.6]{GKRV} we can generalize the notion of twisted D-modules with singular support $\cN$ to not necessarily smooth affine schemes and algebraic stacks.
	\end{rem}
	
	\subsection{Singular support of twisted Betti sheaves.}
	
	Recall (\cite[13.2]{CF}) that for any prestack $X$ locally almost of finite type, and an analytic gerbe $\cG^{\an}$ on $X^{\an}$, we can consider the category $\Shv_{\cG^{\an}}(X)$ of Betti sheaves twisted by $\cG^{\an}$.

	\begin{df}\cite[5.1.2]{KS}
		Let $F \in \Shv_{\cG^{\an}}(S)$ for $S$ a smooth affine scheme of finite type. Let the \emph{micro-support of $F$}, denoted by $\SingS(F)$, be the subset of $T^*S^{\an}$ defined by $p\notin \SingS(F)$ if and only if there exists an open neighborhood $U$ of $p$, such that for any $x_1\in S$ and any real function $\psi$ of class $C^1$ defined in the neighborhood of $x_1$, with $\psi(x_1) = 0$, $d\psi(x_1) \in U$, we have 
		$$({f_{\psi}}_! \circ f_{\psi}^*(F))_{x_1} = 0.$$
		Here $ f_{\psi}$ is the inclusion of the subspace of $x\in S^{\an}$ such that $\psi(x) \geq 0$.
	\end{df}
	
	\begin{rem}
		As in \cite[A.3.3-A.3.6]{GKRV} we can generalize the the notion of micro-support of twisted Betti sheaves to not necessarily smooth affine schemes and algebraic stacks.
	\end{rem}
	
	\begin{lm}\cite[5.4.5]{KS}
		Let $X \rightarrow Y$ be a smooth morphism of manifolds, $\cG$ be an analytic gerbe on $Y$. Let $F \in \Shv_{\cG}(Y)$. Then
		$$\SingS(f^*F) = df \circ f^{-1}(\SingS(F)).$$	
	\end{lm}
	
	\begin{proof}
		The question is local, so it reduces to the untwisted case, which is \cite[Proposition 5.4.5]{KS}.
	\end{proof}
	
	\begin{lm}\cite[5.4.4]{KS}
		Let $X \rightarrow Y$ be a proper morphism of manifolds, $\cG$ be an analytic gerbe on $Y$. Let $F \in \Shv_{f^*\cG}(X)$. Then
		$$\SingS(f_!F) = f\circ df^{-1}(\SingS(F)).$$
	\end{lm}
	
	\begin{proof}
		Again, it reduces to the untwisted case, which is \cite[Proposition 5.4.4]{KS}.
	\end{proof}
	
	\subsubsection{} Recall (\cite[13.2]{CF}) that for any prestack $X$ locally almost of finite type, and an analytic gerbe $\cG^{\an}$ on $X$, we can also consider the small category of algebraically constructible Betti sheaves $\Shv_{\cG^{\an}}(X)^{\operatorname{cnstr}}$ and its ind-completion $\Shv_{\cG^{\an}}(X)$. Recall also that we have \cite[Corollary 13.3.1]{CF} a generalization of the Riemann-Hilbert correspondence to the twisted setting:
	\begin{equation}\label{RH}
	\rh: \Shv_{\cG}(X) \xrightarrow{\sim} \Shv_{\cG^{\an}}(X),
	\end{equation}
	where $X$ is a prestack locally almost of finite type and $\cG$ a tame de Rham gerbe on $X$. 
	
	\begin{lm}
		For $F \in \Shv_{\cG}(X)$ we have 
		$$\SingS(F) = \SingS(\rh(F)).$$
	\end{lm}
	
	\begin{proof}
		The question is local, so the statement reduces to \cite[Theorem 11.3.3]{KS}.
	\end{proof}
	
	\subsection{Twisted miscrostalks.}
	\subsubsection{} Let $Y$ be a real analytic manifold with an analytic gerbe $\cG$, $F: Y \rightarrow \bR$ a smooth function.
	
	\begin{df}
		We define twisted vanishing cycles functor $$\phi^{\cG}_F: \Shv_{\cG}(Y) \rightarrow \Shv_{\cG}(Y^{F=0})$$
		by first $!$-restricting to the locus where $F \geq 0$ and then $*$-restricting to the locus where $F=0$. 
		
		Further $*$-restricting to $y_0 \in Y^{F=0}$ we get $\phi^{\cG}_{F, y_0}$. 
	\end{df}
		Let $\Lambda \subset T^*Y$ be a subanalytic closed conic Lagrangian and $\Shv_{\Lambda, \cG}$ be category of twisted Betti sheaves with singular support in $\Lambda$. Let $p: T^*Y \rightarrow Y$. 
	\begin{pr}\label{microstalk independence}
		For any smooth point $\xi_0 \in \Lambda$ and any function $F$ on $Y$ such that $dF$ intersects $\Lambda$  transversely at $\xi_0$, the functor $\phi^{\cG}_{F, p(\xi_0)}[\frac{\ind}{2}]$ does not depend on the choice of such function (depends only on $\xi_0$). Here $\ind$ is the Maslov index of the triple ($dF$, $\Lambda$, $T^*_{p(\xi_0)}Y$) in the symplectic
		vector space $T_{\xi_0}^* T^*Y$.
	\end{pr} 
	
	\begin{proof}
		Let $j:U \hookrightarrow Y$ be an open neighborhood of $p(\xi_0)$ such that $j^!(\cG)$ is trivial. Note that $$\phi^{\cG}_{F, p(\xi_0)}[\frac{\ind}{2}] \cong \phi^{j^!\cG}_{F|_U, p(\xi_0)}[\frac{\ind}{2}] \circ j^!.$$
		But since $j^!(\cG)$ is trivial we reduced the statement to the case of \cite[Proposition 7.5.3]{KS}. 
	\end{proof}
	
	\begin{df}
		In the setup of Proposition \ref{microstalk independence} we will cal $\phi^{\cG}_{F, p(\xi_0)}[\frac{\ind}{2}]$ the microstalk functional at $\xi_0$ and denote it by $m_{\xi_0}$.
	\end{df}
	
	\subsubsection{}  Now let $Y$ be a complex analytic manifold and $\Lambda \in T^*Y$ is a complex subanalytic closed conic Lagrangian, $\xi_0 \in \Lambda$ a smooth point .
	
	\begin{lm}\label{properties of ms}
		The twisted microstalk at $\xi_0 \in \Lambda$ is t-exact and commutes with Verdier duality.
	\end{lm}
	
	\begin{pr}\label{microstalk_cc}
		For $\cF \in \Shv_{\Lambda, \cG}(Y)^{\heartsuit, \operatorname{cnstr}}$ denote by $$\CC(\cF):= \sum_{\beta \in \Irr(\Lambda)} c_{\beta, \cF} [\beta]$$ the characteristic cycle of $\cF$. Then we have 
		$$\chi(m_{\xi_0}(\cF)) = c_{\beta_{\xi_0},\cF},$$
		where $\xi_0 \in  \beta_{\xi_0}$.
	\end{pr}
	
	\begin{proof}
		Let $j:U \hookrightarrow Y$ be an open neighborhood of $p(\xi_0)$ such that $j^!(\cG)$ is trivial. As we saw in Proposition \ref{microstalk independence} we have 
		$$m_{\xi_0}(\cF) \cong m_{\xi_0}(j^!\cF).$$
		On the other hand, $$ c_{\beta_{\xi_0},\cF} \cong  c_{\beta_{\xi_0},j^!\cF},$$
		so we reduced the statement to the untwisted situation. In that case, by \cite[(9.4.9)]{KS} and \cite[Proposition 6.6.1(ii)]{KS} we get that 
		$$j^{-1}\cF \cong V_{U_{\alpha}}$$
		in $$\Shv_{\Lambda}(U; \xi_0):= \Shv_{\Lambda}(U)/\Shv_{T^*U \setminus \{\xi_0\}}(U)$$ for some analytic submanifold $U_{\alpha}$ of $U$. Since $m_{\xi_0}$ vanishes on $\Shv_{T^*U \setminus \{\xi_0\}}(U)$ we get the result.
	\end{proof}

\section{Whittaker-temperedness.}

\subsection{Local formalism.}
\begin{ntn}
	Let $$\Sph_{\kappa, x}:=\D_{\kappa} (\Gra_G)^{L^+G}$$ be the twisted spherical Hecke category, viewed as a monoidal category via convolution.
\end{ntn}
 Recall that for any $\mathbf{D} \in \DGCat$ with an action of $\D_{\kappa}(LG)$ we can consider
\begin{equation}\label{av}
\Av_!^{\chi}: \mathbf{D}^{L^+G} \xrightarrow{\oblv} \mathbf{D}^{L^+G \cap LN} \xrightarrow{\Av_!^{\chi}} \Whit_{\kappa}(\mathbf{D}).
\end{equation}
Here $$\Av_!^{\chi}: \mathbf{D}^{L^+G \cap LN} \rightarrow \Whit_{\kappa}(\mathbf{D})$$ is left adjoint to the forgetful functor $\Whit_{\kappa}(\mathbf{D}) \rightarrow \mathbf{D}^{L^+G \cap LN}$, which exists by \cite[Theorem 2.3.1]{R}.
\begin{ntn}
	Let $\mathbf{D}^{L^+G, \atemp} := \Ker\Av_!^{\chi} \hookrightarrow \mathbf{D}^{L^+G}$.
\end{ntn}

\begin{ntn}
	Let $I \subseteq G(O)$ (resp. $I^-$) denote the Iwahori subgroup corresponding to $B \subseteq G $ (resp. $B^{-} \subseteq G$).
Let $\stackrel{\circ}{I}$ denote the prounipotent radical of $I$. Let $\stackrel{\circ}{I}_1 := \Ad_{-\check{\rho}(t)}(\stackrel{\circ}{I})$ and let $I^-_1 := \Ad_{-\check{\rho}(t)}(I^-)$. Finally,
we let $\cK_1$ denote $\Ad_{-\check{\rho}(t)}$ applied to the first congruence subgroup of $G(O)$.
\end{ntn}
Note that $\stackrel{\circ}{I}_1 /\cK_1 = N$ and $I^{-1} /\cK_1 = B^-$.

\begin{lm}\label{leftadj}
	The inclusion  $\mathbf{D}^{L^+G, \atemp} \hookrightarrow \mathbf{D}^{L^+G}$ admits a left adjoint. 
\end{lm}
\begin{proof}
	It suffices to show that (\ref{av}) commutes with limits. We have the following commutative diagram:
	\begin{equation}\label{babywhit}
	\begin{tikzcd}
	{\mathbf{D}}^{L^+G}\arrow[drr, "\Av_!^{\chi}" ']\ar[r, rightarrow, "\oblv"' ']&   	{\mathbf{D}}^{I^-} \ar[r, rightarrow, "\Av_*"' ']& {\mathbf{D}}^{I_1^-, \mu_{\kappa}} \arrow[r, equal]& ({\mathbf{D}}^{\cK_1} )^{B^-, \mu_{\kappa}} \ar[r, rightarrow, "\Av_!^{\chi}"' '] &  ({\mathbf{D}}^{\cK_1} )^{N, \chi} \arrow[r, equal]&  \Whit_{\kappa}^{\leq 1}(\mathbf{D}) \arrow[dlll, "\iota_{1, \infty, !}"]\\
	&   & \Whit_{\kappa}(\mathbf{D}).& & & 
	\end{tikzcd}
	\end{equation}
	Here $\mu_{\kappa}$ is the character D-module on $I_1^-$ coming from $$\kappa(\check{\rho}, -): \Lie(I_1^-) \rightarrow \mathfrak{t} \xrightarrow{\kappa(\check{\rho}, -)} k.$$ 
	Note that $\kappa(\check{\rho}, -)$ is trivial on $\cK_1$, hence it descends to a character on $\mathfrak{b}^-$. 
	
	Commutativity of (\ref{babywhit}) follows from the fact that the functor $$\Av_*: {\mathbf{D}}^{I^-} \rightarrow {\mathbf{D}}^{I_1^-, \mu_{\kappa}}$$ is an equivalence, with the inverse the $!$-averaging functor.
	
	Note that $\oblv$ commutes with limits. Indeed, its left adjoint $\Av_!$ is well-defined due to the fact that $L^+G/I$ is proper.   By \cite[3.2.3]{FR}, the functor $\Av_!^{\chi}$ coincides with $\Av_*^{\chi}$ up to a shift, thus commutes with limits. Finally, $\iota_{1, \infty, !}$ commutes with limits by \cite[Theorem 2.3.1 (1)]{R}.
\end{proof}

\begin{df}
	Define $\mathbf{D}^{L^+G, \temp} $ to be the quotient of  $\mathbf{D}^{L^+G}$ by $\mathbf{D}^{L^+G, \atemp} $.
\end{df}

\begin{rem}
	By Lemma \ref{leftadj} we thus have a recollement
	\begin{equation}
	\begin{tikzcd}
	\mathbf{D}^{L^+G, \atemp} \ar[r, hookrightarrow,  shift right, ""]
	& \arrow[l,  shift right, ""]  \mathbf{D} \arrow[r,  shift right, ""]
	& \ar[l, hookrightarrow,  shift right, ""]\mathbf{D}^{L^+G, \temp} .
	\end{tikzcd}
	\end{equation}
	
\end{rem}

\begin{ex}
	$\mathbf{D} = \D_{\kappa}(\Gra_G)$.
\end{ex}

\begin{df}
	For a category $\mathbf{C}$ with an action of $\Sph_{\kappa, x}$ define $$\mathbf{C}^{\temp} := \mathbf{C} \otimes_{\Sph_{\kappa, x}} \Sph_{\kappa, x}^{\temp} \subseteq \mathbf{C},$$
	$$\mathbf{C}^{\atemp} := \mathbf{C} \otimes_{\Sph_{\kappa, x}} \Sph_{\kappa, x}^{\atemp} \subseteq \mathbf{C}.$$
\end{df}

\begin{rem}
	A priori, the notion of Whittaker-temperedness depends on the choice of $x \in X$, but it is expected to be independent of such choice. 
\end{rem}

\subsection{Irregular singular support in finite dimensions and Whittaker averaging.}

Parallel to \cite[3.1]{FR} we set the following notation.

\begin{df}
	For an algebraic stack $\cY$ locally of finite type, equipped with a $G$-action and a $G$-equivariant twisting $\tau$, define $$\D_{\tau, G-\irreg}(\cY) := \D_{\tau, \mu^{-1}(\mathfrak{g}_{\irreg})}(\cY),$$
	where $\mu: T^*\cY \rightarrow \mathfrak{g}^* \cong \mathfrak{g}$ is the moment map. 
	Set $$\Shv_{\tau, G-\irreg}(\cY):= \Shv_{\tau}(\cY) \cap \D_{\tau, G-\irreg}(\cY).$$
\end{df}

\begin{pr}\label{G-irreg zero}
	For any character $D$-module $\mu$ on $B^-$ the finite Whittaker averaging functor 
	$$\Shv_{\tau, G-\irreg}(\cY)^{B^-, \mu} \rightarrow \D_{\tau}(\cY)^{B^-, \mu} \xrightarrow{\Av_!^{\chi}} \D_{\tau}(\cY)^{N, \chi}$$
	is zero.
\end{pr}

Up to some minor changes, our proof will follow the one given in \cite[3.3]{FR}. First, we note that it suffices to prove the assertion for regular holonomic $\cF \in \Shv_{\tau, G-\irreg}(\cY)^{B^-, \mu, \heartsuit}$.
\subsubsection{Reduction to the case $\cY = G \times \cY_0$.} Consider $G \times \cY$  as acted on by $G$ via the action on the first factor. We have the $G$-equivariant map:
$$G \times \cY \xrightarrow{\operatorname{act}} \cY.$$
Since $\operatorname{act}$ is smooth, $\operatorname{act}^!$ is t-exact, conservative, and maps $\Shv_{\tau, G-\irreg}(\cY)^{B^-, \mu}$ to $$\Shv_{\tau, G-\irreg}(G \times \cY)^{B^-, \mu}.$$ Since $\operatorname{act}$ is $G$-equivariant, $\operatorname{act}^!$ commutes with finite Whittaker averaging functor. 

\subsubsection{Reduction to the case $\cY = G $.} Since $\Av_!^{\chi}(\cF)$ is a compact, holonomic object it suffices to check that for every pair $(g, y) \in (G \times \cY_0)(k)$ we have 
$$i_{(g, y)}^!(\Av_!^{\chi}(\cF)) =0.$$
Note that the map $$ \id\times i_y: G \rightarrow G \times \cY_0$$
is $G$-equivariant, therefore
$$(\id\times i_y)^!\Av_!^{\chi} (\cF) \cong \Av_!^{\chi}(\id\times i_y)^!(\cF) \in \D(G)^{N, \chi}.$$
Finally, the fact that $(\id\times i_y)^!(\cF) $ lies in $\Shv_{ G-\irreg}(G)^{B^-, \mu}$ follows from \cite[Proposition 3.3.4.2]{FR}.

\subsubsection{Proof for $\cY = G$.} We are then in the setting $\cF \in \D( G)^{B^-, \mu }$ with $$\SingS(\cF) \subseteq \widetilde{\cN} \times_{\cN} \cN_{\irreg} \subset \widetilde{\cN} \cong T^*(B^- \backslash G),$$ where $\widetilde{\cN}$ is the Springer resolution of the nilpotent cone.  We want to show that 
$$\Av_!^{\chi}(\cF) =0.$$
However, this follows from \cite[Theorem 3.2.4.1]{FR}.

\subsection{Whittaker-temperedness and nilpotent singular support.}\label{temp and nilp}

Recall from \cite[4.4]{FR} the notion of an (irregular) ramified Higgs field. 
\subsubsection{Ramified Higgs fields.} Let $P \subset G(K)$ be a compact open subgroup. Let $(\cP_G, \tau) \in \Bun_G^{P-\lvl}$ be a point. Recall that $T^*_{(\cP_G, \tau)}\Bun_G^{P-\lvl}$ classifies Higgs bundles $\varphi \in \Gamma(X \setminus x, \mathfrak{g}_{\cP_G} \otimes \Omega_X^1)$ with the condition
$$(\cP_G, \tau, \varphi_{D_x^{\circ}}) \in \Lie(P)^{\perp}/P \subseteq \mathfrak{g}((t))dt / P = \mathfrak{g}((t))^{\vee} / P .$$

We say that such a ramified Higgs bundle is \emph{irregular }if the Higgs bundle $(\cP_{G, X \setminus x}, \varphi)$ on $X \setminus x$ is irregular. Note that irregular ramified Higgs bundles form a closed conical substack of $T^*\Bun_G^{P-\lvl}$.
\begin{ntn}
	We denote by $\D_{\kappa, \irreg}(\Bun_G^{P -\lvl})$ and $\Shv_{\kappa, \irreg}(\Bun_G^{P -\lvl})$ the categories of twisted D-modules and sheaves with irregular singular support.
\end{ntn}

\subsubsection{} Note that $\kappa$ trivializes canonically over $L^+G$ and over $\cK_1$. Moreover, the two trivializations coincide on $L^+G \cap \cK_1$.  Therefore it makes sense to formulate the following statement:
\begin{pr}\label{irregular convolution}
	The natural convolution action induces a functor 
	$$\D_{\kappa}(\cK_1 \backslash LG / L^+G) \otimes \D_{\kappa, \irreg}(\Bun_G^{I^- - \lvl}) \rightarrow \D_{\kappa, \irreg}(\Bun_G^{\cK_1 - \lvl}).$$
\end{pr}

\begin{pr}\label{temperedSS}
	Every object of $$\Shv_{\kappa, \Nilp^{\irreg}}(\Bun_G) \subset \Shv_{\kappa, \Nilp}(\Bun_G)$$ is Whittaker-anti-tempered.
\end{pr}

\begin{proof}
	Chasing (\ref{babywhit}) we obtain that it suffices to show that the composition 
	\begin{equation}\label{G-irreg}
	\begin{split}
	\Shv_{\kappa, \Nilp^{\irreg}}(\Bun_G) \subset \D_{\kappa}(\Bun_G) \xrightarrow{\Av_*} \D_{\kappa}(\Bun_G^{\lvl})^{I_1^-, \mu_{\kappa}}=\\ =\D_{\kappa}(\Bun_G^{\cK_1- \lvl})^{B^-, \mu_{\kappa}} \xrightarrow{\Av_!^{\chi}} D_{\kappa}(\Bun_G^{\cK_1 -\lvl})^{N, \chi}
	\end{split}
	\end{equation}
	is zero. 
	
	Note that the composition
	\begin{equation}
		\D_{\kappa}(\Bun_G) \rightarrow \D_{\kappa}(\Bun_G^{\lvl})^{I_1^-, \mu_{\kappa}}= \D_{\kappa}(\Bun_G^{\cK_1- \lvl})^{B^-, \mu_{\kappa}} \xrightarrow{\oblv} D_{\kappa}(\Bun_G^{\cK_1- \lvl})
	\end{equation}
	coincides with 
	$$	\D_{\kappa}(\Bun_G)  \xrightarrow{\Av_*} \D_{\kappa}(\Bun_G^{\cK_1- \lvl}),$$
	which by  Proposition \ref{irregular convolution} maps $\Shv_{\kappa, \Nilp^{\irreg}}(\Bun_G)$ to $\Shv_{\kappa, \irreg}(\Bun_G^{\cK_1 -\lvl})$.  As in \cite[4.5.2]{FR} we obtain that 
	$$\Shv_{\kappa, \irreg}(\Bun_G^{\cK_1 -\lvl}) \subseteq \Shv_{\kappa, G-\irreg}(\Bun_G^{\cK_1 -\lvl}).$$
	
	Hence we can rewrite the composition (\ref{G-irreg}) as 
	\begin{equation}
		\begin{split}
		\Shv_{\kappa, \Nilp^{\irreg}}(\Bun_G) \rightarrow \Shv_{\kappa, \irreg}(\Bun_G^{\cK_1 -\lvl})^{B^-, \mu_{\kappa}} \subseteq \\
		\subseteq \Shv_{\kappa, G-\irreg}(\Bun_G^{\cK_1 -\lvl})^{B^-, \mu_{\kappa}}\rightarrow \D_{\kappa}(\Bun_G^{\cK_1 -\lvl})^{N, \chi},
		\end{split}
	\end{equation}
	which vanishes by Proposition \ref{G-irreg zero}.
\end{proof}

\subsection{Whittaker-tempered category for nice levels.}\label{nice}
The goal of this section is to show that for ``most" rational levels $\kappa = (-\check{h} + \frac{p}{q})\kappa_b$ and for any DG category $\mathbf{C}$ with an action of $\Sph_{\kappa, x}$ one has 
$$\mathbf{C}^{\temp} = \mathbf{C}.$$

\begin{pr}\label{nice levels}
	Assume that the integral Weyl group $W_{\kappa(\check{\rho}, -)}$ is trivial. Then for any $\mathbf{C}$ with an action of $\Sph_{\kappa, x}$ one has 
	$$\mathbf{C}^{\temp} = \mathbf{C}.$$
\end{pr}

\begin{rem}
	Explicitly, the condition in Proposition \ref{nice levels} for a simple group $G$ of type ADE is saying that $q  > h$.
\end{rem}

\begin{proof}
	By definition, it suffices to show that for any $\mathbf{D}$ with an action of $D_{\kappa}(LG)$ the functor 
	$$\Av_!^{\chi}: \mathbf{D}^{L^+G \cap LN} \rightarrow \Whit_{\kappa}(\mathbf{D})$$ 
	is conservative. Using (\ref{babywhit}) we reduce to showing that for any $\mathbf{D}$ with an action of $D(G)$ the functor 
	\begin{equation}
	\Av_!^{\chi}: \mathbf{D}^{B^-, \mu_{\kappa}} \rightarrow \mathbf{D}^{N, \chi}
	\end{equation}
	is conservative. It is enough to check the universal case $\mathbf{D} = \D(G)$. Moreover, taking weak $G$-invariants on the left is an invertible functor, hence we may take $\mathbf{D} = \mathfrak{g}\Mmod$.
	But then $$\mathbf{D}^{B^-, \mu_{\kappa}} \cong \oplus_{\nu \in \kappa(\check{\rho}, -) + \Lambda} \cO_{\nu},$$
	and the condition that $W_{\nu}$ is trivial is precisely the necessary and sufficient condition for conservativity of $\Av_!^{\chi}$ on $\cO_{\nu}$ by \cite{Los}.
\end{proof}

\section{Generalized Kostant slices.}

\subsection{$D$-Kostant slice.} Let $D$ be an $\check{\Lambda}^+$-valued divisor on $X$. In this subsection  we introduce a generalization of the global Kostant section, which will depend on $D$ and will be denoted by $\Kos_D$. 

\subsubsection{} Recall from \cite[Lemma 2.7.2]{CPSI} that a map $X \rightarrow \bA^1/\bG_m$ with the inverse image of the open point dense in $X$ is equivalent to the data of an effective Cartier divisor on $X$. Therefore we get that the scheme $$\Maps(X, \pt \subseteq \frac{\mathfrak{g_{\geq -1}}/\mathfrak{b} }{T}),$$ where $\pt \subset  \frac{\mathfrak{g_{\geq -1}}/\mathfrak{b} }{T}$ is the open point,  parameterizes  $\check{\Lambda}_{G}^+$-valued divisors on $X$.

\begin{ntn}
	Let subscript $X$ stand for a constant stack over $X$.
\end{ntn}
 
\begin{df}
	We define the stack $\Kos_D$ as the fiber product 
		\[
	\begin{tikzcd}
	\Kos_D\ar[r, rightarrow, ""']\ar[d, rightarrow, ""']&   	\Maps(X, (\frac{f + \mathfrak{b}}{N \rtimes\bG_m})_X \subseteq (\frac{\mathfrak{g}_{\geq -1}}{B \times \bG_m})_X) \times_{\Maps(X, \pt/\bG_m) }\{ \omega^{\frac{1}{2}}\}\ar[d, rightarrow, ""']  \\
	\{D\}\ar[r, rightarrow, ""']&  \Maps(X, \pt_X \subseteq (\frac{\mathfrak{g_{\geq -1}}/\mathfrak{b} }{T})_X),
	\end{tikzcd}
	\]
	where the action of $\bG_m$ on $\mathfrak{g}_{\geq -1}$ is given by $z \cdot y = z^2y$ for $z \in \bG_m$ and $y \in \mathfrak{g}_{\geq -1}$, the homomorphism 
	$$N \rtimes\bG_m \rightarrow B \times \bG_m$$
	is defined as $n \rtimes z \mapsto (n \cdot 2\check{\rho}(z), z)$,
	and the map $$B \times \bG_m \rightarrow T$$ is given by $$B \times \bG_m \xrightarrow{p \times 2\check{\rho}} T \times T \xrightarrow{m} T.$$
\end{df}

\begin{rem}\label{kosD points}
	Let us spell out explicitly what the points of $\Kos_D$ are. The stack $\Kos_D$
	parameterizes pairs $(\cF_B, s)$, where
	\begin{itemize}
		\item a $B$-bundle $\cF_B$, such that the induced $T$-bundle $\cF_T$ is $\check{\rho}(\omega)(-D)$,
		\item an element $s \in \Gamma(X, \omega \otimes (\cF_B \times^B \mathfrak{g}_{\geq -1}))$,
	\end{itemize}
	 such that the projection of $s$ along $$\Gamma(X, \omega \otimes( \cF_B \times^B \mathfrak{g}_{\geq -1})) \rightarrow \Gamma(X, \omega \otimes (\cF_T \times^T \mathfrak{g}_{\geq -1}/\mathfrak{b})) \cong \Gamma(X, \oplus_{i \in \cI} \cO(\check{\alpha}_i(D)))$$
	is the canonical section. 
\end{rem}

\begin{ex}\label{kosD SL2}
	For $G = \PGL_2$ the field-valued points of the stack $\Kos_D$ are the data (up to tensoring by a line bundle) of 
	\begin{itemize}
		\item 	2-dimensional vector bundles $\cE$,
		\item a filtration 
		c
		\item a Higgs field $\theta: \cE \rightarrow \cE \otimes \omega$,
	\end{itemize}
	such that $D$ is the divisor of zeroes of the map 
	$$\cE_1 \rightarrow \cE \xrightarrow{\theta} \cE\otimes \omega \rightarrow \cE_2.$$
\end{ex}

\subsubsection{} One of the important inputs in the main result of this paper is geometry of the intersection  $$\Kos_D \times_{T^* \Bun_G} \Nilp^{\reg}.$$ Let us continue Example \ref{kosD SL2} and describe points of this intersection in the case $G = \PGL_2$.
\subsubsection{} Recall from Proposition \ref{2.5.4.1} that irreducible components of $\Nilp^{\reg}$ are parameterized by a sublattice $$\check{\Lambda}^{\operatorname{rel}} \subset  \check{\Lambda}.$$ 

\begin{ex}
For $G = \PGL_2$ and $\check{\lambda} \in \check{\Lambda}^{\operatorname{rel}}$, points of $\widetilde{\Nilp}^{\reg, \check{\lambda}} $ are given by the data (up to tensoring by a line bundle) of
\begin{itemize}
	\item 2-dimensional vector bundles $\cE$,
	\item short exact sequences 
	$\cL_1 \rightarrow \cE \rightarrow \cL_2$ with $\deg \cL_1 - \deg \cL_2 = \check{\lambda}$,
	\item a non-zero map $\bar{\phi}: \cL_2 \rightarrow \cL_1 \otimes \omega$, such that the corresponding Higgs field is 
	$$\phi: \cE \rightarrow \cL_2 \xrightarrow{\bar{\phi}} \cL_1 \otimes \omega \rightarrow \cE \otimes \omega$$
\end{itemize}
\end{ex}

\begin{ex}
	For $G=\PGL_2$ the field-valued points of the stack $\Kos_D \times_{T^* \Bun_G} \Nilp^{\reg}_{\redu}$ are given by the data (up to tensoring by a line bundle) of
	\begin{itemize}
		\item 2-dimensional vector bundles $\cE$,
		\item filtration $\cL_1 \rightarrow \cE \rightarrow \cL_2$,
		\item a non-zero map $\bar{\phi}: \cL_2 \rightarrow \cL_1 \otimes \omega$, such that the corresponding Higgs field is 
		$$\phi: \cE \rightarrow \cL_2 \xrightarrow{\bar{\phi}} \cL_1 \otimes \omega \rightarrow \cE \otimes \omega$$
		\item filtration $$\cE_1\rightarrow \cE \rightarrow\cE_2,$$such that $D$ is the divisor of zeroes of the map 
		$$\cE_1 \rightarrow \cE \xrightarrow{\theta} \cE\otimes \omega \rightarrow \cE_2.$$
	\end{itemize}
	This implies that a point of $\Kos_D \times_{T^* \Bun_G} \Nilp^{\reg}_{\redu}$ gives a commutative diagram 
		\[
	\begin{tikzcd}
	&   & 0\ar[d, rightarrow, ""']& & \\
	&   & \cL_1 \ar[d, rightarrow, ""']\ar[dr, rightarrow, "\neq 0"']& &  \\
	0\ar[r, rightarrow, ""']&   \cE_1 \ar[r, rightarrow, ""']  & \cE \ar[r, rightarrow, ""']\ar[d, rightarrow, ""']& \cE_2\ar[r, rightarrow, ""']& 0\\
	&   & \cL_2\ar[d, rightarrow, ""']& & \\
	&   & 0.& & 
	\end{tikzcd}
	\]
\end{ex}

\subsubsection{} 	Comparing to \cite[Example 2.11.2] {CPSI} we see that $\Kos_D \times_{T^* \Bun_G} \Nilp^{\reg}_{\redu}$ is connected to Zastava spaces. 

\subsection{Zastava spaces.}
In this subsection we digress to review the definition and properties of the relevant Zastava space introduced in \cite{BFGM}.

\begin{df}
	Let $\stackrel{\circ}{\cZ}$ denote the stack $$\Maps(X, \pt/T \subseteq (\pt/B \times_{\pt/G} \pt/B)) \times_{\Bun_T} \{ \check{\rho}(\omega)(-D) \},$$
	where $\pt/T$ embeds into $ \pt/T \subseteq (\pt/B \times_{\pt/G} \pt/B)$ as an open Bruhat cell and the map $$\Maps(X, \pt/T \subseteq (\pt/B \times_{\pt/G} \pt/B)) \rightarrow {\Bun_T}$$
	is induced by projection on the first coordinate 
	$$\pt/B \times_{\pt/G} \pt/B \rightarrow \pt/B \rightarrow \pt/T.$$
\end{df}
We have an open embedding $$\stackrel{\circ}{\cZ} \rightarrow \Bun_N^{\omega(-D)} \times_{\Bun_G} \Bun_B.$$ 
Also, as in \cite[2.11]{CPSI}, we have a map 
$$\stackrel{\circ}{\pi} : \stackrel{\circ}{\cZ} \rightarrow \Div^{\check{\Lambda}^+}_{\eff},$$
where $\Div^{\Lambda^+}_{\eff}$ is the scheme of $\check{\Lambda}^+$-divisors on $X$. let $\Div^{\check{\lambda}}_{\eff}$ denote the connected component corresponding to $\check{\lambda} \in \check{\Lambda}^+$. Let $$\stackrel{\circ}{\pi}^{\check{\lambda}} : \stackrel{\circ}{\cZ}_{\check{\lambda}} \rightarrow \Div^{^{\check{\lambda}}}_{\eff}$$
denote the fiber of $\stackrel{\circ}{\cZ}$ over $ \Div^{^{\check{\lambda}}}_{\eff}$. As in \cite{BFGM} and \cite[2.11]{CPSI}, we have that $\stackrel{\circ}{\cZ}_{\check{\lambda}}$ is a smooth variety.

\begin{ex}\label{pt}
	For $\check{\lambda}=0$ we have $\stackrel{\circ}{\cZ}_{0} = \pt$.
\end{ex}

\subsection{Intersection of $D$-Kostant slice and $\Nilp^{\reg}_{\redu}$ vs Zastava space.}

Our first goal is to construct a map $$\iota_D: \Kos_D \times_{T^* \Bun_G} \widetilde{\Nilp}^{\reg} \rightarrow \stackrel{\circ}{\cZ},$$
where $$\widetilde{\Nilp}^{\reg} := \Maps(X, \frac{\stackrel{\circ}{\mathfrak{n}}}{B \times \bG_m} \subseteq \frac{\mathfrak{n}}{B\times \bG_m})\times_{\Maps(X, \pt / \bG_m)} \{ \omega\}.$$

\begin{cnstr}
	By definition, $\Kos_D \times_{T^* \Bun_G} \widetilde{\Nilp}^{\reg}$ is given by 
	\begin{equation}
	\begin{split}
	\Maps\Biggl(X,  \Bigl(\frac{f+\mathfrak{b}}{N \rtimes \bG_m} \times_{\frac{\mathfrak{g}}{G \times \bG_m}} \frac{\stackrel{\circ}{\mathfrak{n}}}{B \times \bG_m}\Bigr)_X \subseteq \{D\}_X \times_{(\frac{\mathfrak{g}_{\geq -1}/\mathfrak{b}}{T})_X}\Bigl(\frac{\mathfrak{g}_{\geq -1}}{B \times \bG_m} \times_{\frac{\mathfrak{g}}{G \times \bG_m}} \frac{\mathfrak{n}}{B\times \bG_m}\Bigr)_X \Biggr)\times_{\Bun_{\bG_m}}  \{ \omega\}.
	\end{split}
	\end{equation}

	This obviously maps to $$\Maps(X, \{ \check{\rho}(\omega)(-D)\}_X\times_{(\pt/T)_X} (\pt/B \times_{\pt/G} \pt/B)_X  ),$$
	but we need to check that the image generically lands in the open stratum $$\pt \subset_X \{ \check{\rho}(\omega)(-D)\}_X\times_{(\pt/T)_X} (\pt/B \times_{\pt/G} \pt/B)_X .$$ 
	It suffices to show that the composition
	$$\frac{f+\mathfrak{b}}{N } \times_{\frac{\mathfrak{g}}{G}} \frac{\mathfrak{\mathfrak{n}^\circ}}{B} \subset  \frac{\mathfrak{g}_{\geq -1}}{B } \times_{\frac{\mathfrak{g}}{G}} \frac{\mathfrak{n}}{B} \rightarrow \pt/B \times_{\pt/ G} \pt/B$$
	factors through the open orbit $$\frac{\pt}{T}\subset (\pt/B \times_{\pt/G} \pt/B).$$ However, 
	$$\frac{f+\mathfrak{b}}{N } \times_{\frac{\mathfrak{g}}{G}} \frac{\mathfrak{\mathfrak{n}^\circ}}{B} \cong \pt,$$
	and one can explicitly describe this point and show that the image lands in the open stratum $\frac{\pt}{T}$.
\end{cnstr}

\begin{lm}\label{intersection smooth}
	The map $\iota_{D}$ is schematic. In particular, since $\stackrel{\circ}{\cZ}$ is a scheme, we get that $$\Kos_D \times_{T^* \Bun_G} \widetilde{\Nilp}^{\reg}$$ is a scheme.
\end{lm}

\begin{proof}
		Since $X$ is projective, it suffices to show that the map $$\{D\}_X \times_{(\frac{\mathfrak{g}_{\geq -1}/\mathfrak{b}}{T})_X}\Bigl(\frac{\mathfrak{g}_{\geq -1}}{B} \times_{\frac{\mathfrak{g}}{G }} \frac{\mathfrak{n}}{B}\Bigr)_X \rightarrow \{ \cO(-D)\}_X\times_{(\pt/T)_X} (\pt/B \times_{\pt/G} \pt/B)_X $$ is quasi-projective schematic. However, note that 
		
		$$\{D\}_X \times_{(\frac{\mathfrak{g}_{\geq -1}/\mathfrak{b}}{T})_X}(\frac{\mathfrak{g}_{\geq -1}}{B})_X \rightarrow \{\cO(-D)\}_X\times_{(\pt/T)_X} (\frac{\mathfrak{g}_{\geq -1}}{B})_X \rightarrow  \{\cO(-D)\}_X\times_{(\pt/T)_X} (\pt/B)_X$$
		is quasi-projective schematic since the first and the second maps are quasi-projective schematic. Also, 
		$$\frac{\mathfrak{g}}{G } \rightarrow \pt/G,$$
		$$\frac{n}{B} \rightarrow \pt/B$$
		are quasi-projective schematic, hence the result.


\end{proof}

\begin{pr}\label{intersection and zastava closed embedding}
	The map $$\iota_{D}: \Kos_D \times_{T^* \Bun_G} \widetilde{\Nilp}^{\reg} \rightarrow \stackrel{\circ}{\cZ}$$ is a closed embedding.
\end{pr}

\begin{proof}
	
	We first show that $\iota_D$ is proper. We will do so by checking the valuative criterion, following an argument of \cite[Appendix A]{FKM}. 
	Denote by $\xi$ the generic point of $X$. Let $\cD$ be the spectrum of a discrete valuation ring, and let $\stackrel{\circ}{\cD} \subset \cD$  be the spectrum of its fraction field. 
	
	To unburden the notation, denote 
	$$ \{ \omega\}\times_{{\frac{\pt}{\bG_m}}_X} \Bigl(\frac{f+\mathfrak{b}}{N \rtimes \bG_m} \times_{\frac{\mathfrak{g}}{G \times \bG_m}} \frac{\stackrel{\circ}{\mathfrak{n}}}{B \times \bG_m}\Bigr)_X \subseteq  \biggl(\{ D , \omega\}\times_{(\frac{\mathfrak{g_{\geq -1}}/\mathfrak{b}}{T \times}  \times \pt / \bG_m)_X} \Bigl(\frac{\mathfrak{g}_{\geq -1}}{B \times \bG_m} \times_{\frac{\mathfrak{g}}{G \times \bG_m}} \frac{\mathfrak{n}}{B\times \bG_m}\Bigr)_X\biggr) $$
	by $\stackrel{\circ}{\cY_s} \subseteq \cY_s,$
	and 
	$$\pt\subset \{ \check{\rho}(\omega)(-D) \}\times_{\frac{\pt}{T}_X} (\pt/B \times_{\pt/G} \pt/B)_X$$
	by $\stackrel{\circ}{\cY_t} \subseteq \cY_t.$ Given a solid commutative cube
	\begin{equation}\label{cube}
	\begin{tikzcd}[row sep=1.5em, column sep = 1.5em]
	\stackrel{\circ}{\cD}\times \xi \arrow[rr] \arrow[dr, ,""] \arrow[dd,swap] &&
	\stackrel{\circ}{\cY_s} \arrow[dd, "\simeq"] \arrow[dr,""] \\
	& \stackrel{\circ}{\cD}\times X \arrow[dd] \arrow[rr] &&
	\cY_s\arrow[dd,""] \\
	\cD \times \xi \arrow[rr,] \arrow[dr, ""] && \stackrel{\circ}{\cY_t} \arrow[dr] \\
	& \cD \times X \arrow[rr] \arrow[rruu, dashed]&& \cY_t
	\end{tikzcd}
	\end{equation}
	we need to produce the dashed arrow. Note that since $ \stackrel{\circ}{\cY_s} \rightarrow  \stackrel{\circ}{\cY_t}$ is an isomorphism we get a map $\cD \times \xi \rightarrow \stackrel{\circ}{\cY_s}$.
	
	Let us spell out the data encoded by the map $\cD \times X  \rightarrow \cY_t$ in (\ref{cube}). We are given a $G$-bundle $\cF$ on $\cD \times X$ with two $B$-reductions $\cF_B^1, \cF_B^2$ with the condition that the $T$-bundle $\cF_T^1$ induced from $\cF_B^1$ is $\check{\rho}(\omega)(-D)$. The desired lift $\cD \times X \rightarrow {\cY_s}$ encodes the data of equivariant maps 
	\begin{itemize}
		\item$\cF_B^1 \times_X \Tot(\omega) \rightarrow \mathfrak{g_{\geq -1}},$
		\item $\cF_B^2 \times_X \Tot(\omega) \rightarrow \mathfrak{n},$
		\item $\cF\times_X \Tot(\omega) \rightarrow \mathfrak{g},$
	\end{itemize}
	such that the induced map $$\cF_T^1 \rightarrow \mathfrak{g}_{\geq -1} / \mathfrak{b}$$
	corresponds to $\{D\} \in \Maps(X, \frac{\mathfrak{g}_{\geq -1} / \mathfrak{b}}{T})$.

	But we have two compatible systems of maps: 
		\begin{itemize}
		\item$(\cF_B^1 \times_X \Tot(\omega))_{\cD \times \xi} \rightarrow \mathfrak{g_{\geq -1}},$
		\item $(\cF_B^2 \times_X \Tot(\omega))_{\cD \times \xi}  \rightarrow \mathfrak{n},$
		\item $(\cF\times_X \Tot(\omega) )_{\cD \times \xi} \rightarrow \mathfrak{g},$
	\end{itemize}
and 
		\begin{itemize}
		\item$(\cF_B^1 \times_X \Tot(\omega))_{\stackrel{\circ}{\cD} \times X} \rightarrow \mathfrak{g_{\geq -1}},$
		\item $(\cF_B^2 \times_X \Tot(\omega))_{\stackrel{\circ}{\cD} \times X}  \rightarrow \mathfrak{n},$
		\item $(\cF\times_X \Tot(\omega) )_{\stackrel{\circ}{\cD} \times X} \rightarrow \mathfrak{g}$
	\end{itemize}
	satisfying the above condition. Combined they give 
	\begin{itemize}
		\item$(\cF_B^1 \times_X \Tot(\omega))_{U} \rightarrow \mathfrak{g_{\geq -1}},$
		\item $(\cF_B^2 \times_X \Tot(\omega))_{U}  \rightarrow \mathfrak{n},$
		\item $(\cF\times_X \Tot(\omega) )_{U} \rightarrow \mathfrak{g}.$
	\end{itemize}
	for some dense open $U \subset \cD \times X$ such that $\codim_{\cD \times X}(\cD \times X \setminus U) = 2$. Then since the targets are affine and sources are normal we get the desired maps. The fact that the first map satisfies the above condition follows from the fact that it does so when restricted to $U$. 
	
	To prove that $\iota_{D}$ is a closed embedding it is left to show that it is a monomorphism. Let us check that $\iota_{D}$ is injective on $S$-points for some test affine scheme $S$. Suppose that for some affine $S$ the map on $S$-points is not injective. This means that there exists a triple $\{ \cF, \cF_B^1, \cF_B^2\}$ such that $$\cF_T^1 \cong \check{\rho}(\omega)(-D)$$ with two sets of Higgs fields:
		\begin{itemize}
		\item$\theta^1: \cF_B^1 \times_{S \times X} \Tot(\omega) \rightarrow \mathfrak{g_{\geq -1}},$
		\item $\theta^2: \cF_B^2 \times_{S \times X}\Tot(\omega) \rightarrow \mathfrak{n},$
		\item $\theta: \cF\times_{S \times X} \Tot(\omega) \rightarrow \mathfrak{g},$
	\end{itemize}
	and 
			\begin{itemize}
			\item$\eta^1: \cF_B^1 \times_{S \times X} \Tot(\omega) \rightarrow \mathfrak{g_{\geq -1}},$
			\item $\eta^2: \cF_B^2 \times_{S \times X} \Tot(\omega) \rightarrow \mathfrak{n},$
			\item $\eta: \cF\times_{S \times X} \Tot(\omega) \rightarrow \mathfrak{g},$
			\end{itemize}
	satisfying the above condition. Let us show that $\theta - \eta$, $\theta^i - \eta^i$ are trivial Higgs fields. It suffices to show that they are trivial generically on $S \times X$. But note that the triple $\{\theta^1 - \eta^1, \theta^2 - \eta^2, \theta - \eta\}$ gives a map 
	$$S \times X \rightarrow \frac{\mathfrak{b}}{B} \times_{\frac{\mathfrak{g}}{G}} \frac{\mathfrak{n}}{B},$$
	and since generically the two $B$-reductions are transversal, we get that the Higgs field is generically zero.
\end{proof} 

\begin{ntn}
	Let $\widetilde{\Nilp}^{\reg, \check{\varphi}}:= \widetilde{\Nilp}^{\reg} \times_{\Bun_T} \Bun_T^{\check{\varphi}}$. 
\end{ntn}
Note that $$\widetilde{\Nilp}^{\reg, \check{\varphi}} \rightarrow \Nilp^{\reg}$$ is a locally closed embedding. 
Recall (e.g. from Proposition \ref{2.5.4.1}) that $\widetilde{\Nilp}^{\reg, \check{\varphi}}$ smooth and connected of dimension $\dim(\Bun_G)$.

\begin{ntn}
	We denote by $\Nilp^{\reg, \check{\varphi}}$ the closure of $\widetilde{\Nilp}^{\reg, \check{\varphi}}$ in $\Nilp^{\reg}_{\redu}$.
\end{ntn}

\begin{lm}\label{non-empty intersections}
	The scheme $\Kos_D \times_{T^* \Bun_G} \widetilde{\Nilp}^{\reg, \check{\varphi}}$ is non-empty only if 
	$$(2-2g)\check{\rho} \leq \check{\varphi}\leq (2-2g)\check{\rho} + \deg(D).$$
	Moreover, for $\check{\varphi} = (2-2g)\check{\rho} + \deg(D)$ the scheme $\Kos_D \times_{T^* \Bun_G} \widetilde{\Nilp}^{\reg, \check{\varphi}}$ is a reduced point.
\end{lm}

\begin{proof}
	By Proposition \ref{2.5.4.1} for $(2-2g)\check{\rho} > \check{\varphi}$ we get that  $$\widetilde{\Nilp}^{\reg, \check{\varphi}} = \emptyset,$$
	so $(2-2g)\check{\rho} \leq \check{\varphi}$. 
	
	Let us introduce another notation for the stack $$\Maps(X, \pt/T \subseteq \pt/B \times_{\pt/G} \pt/B) \cong \Maps(X, \pt/T \subseteq \pt/B \times_{\pt/G} \pt/B^-).$$ We will also call it $(\Bun_B \times_{\Bun_{G}} \Bun_{B^-})^0$. Note that
	$${\iota_{D}}|_{\Kos_D \times_{T^* \Bun_G} \widetilde{\Nilp}^{\reg, \check{\varphi}}}: \Kos_D \times_{T^* \Bun_G} \widetilde{\Nilp}^{\reg, \check{\varphi}} \hookrightarrow \stackrel{\circ}{\cZ}$$
	factors through $$(\Bun_B^{(2g-2)\check{\rho} - \deg(D)} \times_{\Bun_{G}} \Bun_{B^-}^{-\check{\varphi}})^0 \times_{\Bun_B^{(2g-2)\check{\rho}  - \deg(D)}} \{\check{\rho}(\omega)(-D)\} \hookrightarrow  \stackrel{\circ}{\cZ} .$$
	By \cite{BFGM}[Proposition 3.2] we have
	$$(\Bun_B^{(2g-2)\check{\rho} - \deg(D) } \times_{\Bun_{G}} \Bun_{B^-}^{-\check{\varphi}})^0 \times_{\Bun_B^{(2g-2)\check{\rho}  - \deg(D)}} \{\check{\rho}(\omega)(-D)\}  \cong \stackrel{\circ}{\cZ}_{-\check{\varphi} - (2g-2)\check{\rho}  - \deg(D)}$$
	for $\check{\varphi} \leq (2-2g)\check{\rho} +\deg(D).$
	Then the second assertion follows from Example \ref{pt}.
	Finally, it is left to check that for $\check{\varphi} >(2-2g)\check{\rho} +\deg(D)$ one has  that
	\begin{equation}\label{Pl}
		(\Bun_B^{(2g-2)\check{\rho} - \deg(D) } \times_{\Bun_{G}} \Bun_{B^-}^{-\check{\varphi}})^0 \times_{\Bun_B^{(2g-2)\check{\rho}  - \deg(D)}} \{\check{\rho}(\omega)(-D)\}
	\end{equation}
		is empty. Indeed, using Plucker description of $\Bun_B$ and $\Bun_{B^-}$, the data of a point of (\ref{Pl}) gives a non-zero map 
		$$\cL^{\lambda}_{\cF_T^{(2g-2)\check{\rho}  - \deg(D)}} \hookrightarrow \cV^{\lambda}_{\cF_G} \twoheadrightarrow \cL^{\lambda}_{\cF_T^{-\check{\varphi}}}$$
		for each dominant weight $\lambda$. Here $\cF_T^{(2g-2)\check{\rho}  - \deg(D)}$ and $\cF_T^{-\check{\varphi}}$ are elements of $\Bun_T^{(2g-2)\check{\rho}  - \deg(D)}$ and $\Bun_T^{-\check{\varphi}}$ respectively. But this implies that $\check{\varphi} \leq (2-2g)\check{\rho} +\deg(D).$
\end{proof}

\begin{ex}\label{lambda_D}
	Denote the point $$ \Kos_D \times_{T^* \Bun_G} {\widetilde{\Nilp}}^{\reg, (2-2g)\check{\rho} + \deg(D)}$$ by $\lambda_D$. Let us give an explicit description of $\lambda_D \in T^*\Bun_G$. Its underlying $G$-bundle is induced from $\check{\rho}(\omega)(-D)$ via $T \subset G$. The Higgs field is given by the image of the canonical element $$\sigma \in \Gamma(X, \omega \otimes(\check{\rho}(\omega)(-D) \times^T \mathfrak{n}))$$ along $$\Gamma(X, \omega \otimes(\check{\rho}(\omega)(-D)\times^T \mathfrak{n})) \rightarrow \Gamma(X, \omega \otimes(\check{\rho}(\omega)(-D)\times^T \mathfrak{g})).$$
\end{ex}

\begin{pr}\label{nilpintersections}
	If the intersection of ${\Nilp}^{\reg, \check{\mu}}$ and $\widetilde{\Nilp}^{\reg, \check{\lambda}}$ inside of ${\Nilp}^{\reg}_{\redu}$ is non-empty, then $\check{\mu} \leq \check{\lambda}$.
\end{pr}

\begin{proof}
	Let $S$ be a spectrum of a DVR with generic point $\eta$ and closed point $s$. We need to show that for a morphism 
	$$f: S \rightarrow {\Nilp}^{\reg}$$ such that 
	$f_{\eta} \in \widetilde{\Nilp}^{\reg, \check{\mu}}$ and $f_s \in \widetilde{\Nilp}^{\reg, \check{\lambda}}$ we have $\check{\mu} \leq \check{\lambda}$.
	
	Consider the commutative diagram 
		\[
	\begin{tikzcd}
	\widetilde{\Nilp}^{\reg}\ar[r, rightarrow, ""']\ar[d, rightarrow, "\iota_{\cN}"']&   	\Bun_B \ar[d, rightarrow, "\iota"']  \\
	{\Nilp}^{\reg}\ar[r, rightarrow, ""']&  \Bun_G^{B-\operatorname{gen}},
	\end{tikzcd}
	\]
	where $\Bun_G^{B-\operatorname{gen}} := \Maps(X, \pt/B \rightarrow \pt/G),$ i.e. the moduli stack of $G$-bundles with generic $B$-reduction. Note that $\iota$ is a bijection on field-valued points. For a field-valued point $p \in \Bun_G^{B-\operatorname{gen}}$, we will say that $p$ is of {\it degree} $\check{\alpha}$ if the corresponding point of $\Bun_B$ lies in $\Bun_B^{\check{\alpha}}$. Note that the morphism $${\Nilp}^{\reg} \rightarrow \Bun_G^{B-\operatorname{gen}}$$ maps field-valued points of $\widetilde{\Nilp}^{\reg, \check{\alpha}} \subset {\Nilp}^{\reg}$ to points of degree $ \check{\alpha}$.
	
	Hence it suffices to show that for the induced morphism 
	$$f: S \rightarrow {\Nilp}^{\reg} \rightarrow \Bun_G^{B-\operatorname{gen}}$$
	we have $\deg(f_{\eta}) \leq \deg(f_s)$. We obtain the diagram 
		\[
	\begin{tikzcd}
	\eta \ar[r, rightarrow, ""']\ar[dd, rightarrow, ""'] & \Bun_B^{\check{\mu}} \ar[d, rightarrow, ""']\\
		&   	\overline{\Bun}_B \times_{\Bun_T} {\Bun_T}^{\check{\mu}} \ar[d, rightarrow, ""']  \\
	S\ar[r, rightarrow, "f"'] \ar[ru, rightarrow, ""']&  \Bun_G^{B-\operatorname{gen}},
	\end{tikzcd}
	\]
	where $\overline{\Bun}_B$ as in \cite[1.2.1]{BG}. The map $f$ lifts to $$S \rightarrow \overline{\Bun}_B \times_{\Bun_T} {\Bun_T}^{\check{\mu}}$$ since $\overline{\Bun}_B \rightarrow \Bun_G^{B-\operatorname{gen}}$ is proper. Finally, note that field-valued points in the image of $$	\overline{\Bun}_B \times_{\Bun_T} {\Bun_T}^{\check{\mu}}$$ in $\Bun_G^{B-\operatorname{gen}}$ have degree $\geq \check{\mu}.$
\end{proof}

\begin{ntn}
	Denote by $\underline{\Nilp}_{\redu}^{\reg, \check{\lambda}}$ the union of $\Nilp^{\reg, \check{\mu}}$ such that $\check{\mu}$ is not less than $\check{\lambda}$.
Set
$$\Nilp_{\redu}^{\reg, <\check{\lambda}} =\Nilp_{\redu}^{\reg} \setminus \underline{\Nilp}_{\redu}^{\reg, \check{\lambda}}.$$
\end{ntn}

\begin{rem}
	By Lemma \ref{nilpintersections} the map $$\underline{\Nilp}_{\redu}^{\reg, \check{\lambda}} \rightarrow \Nilp_{\redu}^{\reg}$$ is a closed embedding.
\end{rem}

\begin{ntn}\label{Nilpgeq}
	Let $$\underline{\Nilp}_{\redu}^{\check{\lambda}}:=\Nilp_{\redu} \setminus \Nilp_{\redu}^{\reg, <\check{\lambda}} \hookrightarrow \Nilp_{\redu}.$$
\end{ntn}

\section{Twisted Whittaker coefficients as microstalks.}\label{NT}

\subsection{Twisted version of Nadler-Taylor theorem.} Let $Y$ be a scheme with a $\bG_m$-action, denote by $Y^0$ the fixed locus.  Let $y_0$ be a fixed point and $i: Y^{>0}_{y_0} \rightarrow Y$ the attracting locus of $y_0$, and let $Y^{\leq 0}$ be the repelling locus. Let $\cL$ be a $\bG_m$-equivariant line bundle on $Y$ such that $i^*\cL$ is canonically trivial. Let  $\widetilde{Y}$ is the total space of $\cL$ with the zero section removed. Let $\cG$ be a twisting on $Y$, such that its pullback to $\widetilde{Y}$ is canonically trivial. 

\subsubsection{} Let $f: Y^{>0}_{y_0} \rightarrow \bA^1$ be a $\bG_m$-equivariant function of $Y^{>0}_{y_0}$, where the action of $\bG_m$ on $\bA^1$ is linear with some weight. Using local coordinates choose (a germ near $y_0$ of) a real-valued smooth function
$$F: Y \rightarrow \bR,$$
such that 
\begin{enumerate}
	\item $F|_{Y^{>0}_{y_0} } = \re f$
	\item $F|_{Y^{\leq 0} \setminus y_0} < 0$.
\end{enumerate}

\begin{thm}\cite[Theorem 2.2.2, Proposition 2.3.1]{NT}\label{extension of NT}
	Assume that $\cL$ has a property that the diagram 
	\[
\begin{tikzcd}
(\widetilde{Y})^{0}\ar[r, rightarrow, ""']\ar[d, rightarrow, ""']&   	\widetilde{Y}\ar[d, rightarrow, "\pi"']  \\
Y^{0}\ar[r, rightarrow, ""']&   Y
\end{tikzcd}
\]
is Cartesian, where $\widetilde{Y}$ is the total space of $\cL$ with the zero section removed. Let $\widetilde{F} = \pi \circ F$ and choose a lift $\widetilde{y}_0 \in (\widetilde{Y})^{0}$ of $y_0$. Then there is a canonical isomorphism of functors 
\begin{equation}\label{extension of NT functors}
	\phi_{f, y_0} \circ i^! \xrightarrow{\sim} \phi^{\cG}_{F, {y}_0} : \Shv_{\cG, \Lambda}^{\bG_m-\eq}(Y) \rightarrow \Vect.
\end{equation}

Moreover, if the shifted conormal $T^*_{Y^{>0}_{y_0}} Y +df$ intersects $\Lambda$ cleanly along smooth locus and dimension of the clean intersection is bounded above by the dimension of $Y^{\leq 0}$, then we can choose $F$ such that $\phi^{\cG}_{F, {y}_0}$ is a shifted twisted microstalk. 
\end{thm}

\begin{proof}
	Let us construct (\ref{extension of NT functors}). First we will show that the natural map 
	\begin{equation}\label{attracktors to points}
		(\widetilde{Y})^{>0}_{\widetilde{y}_0} \xrightarrow{} Y^{>0}_{y_0}.
	\end{equation}
	is an isomorphism. It suffices to check that (\ref{attracktors to points}) is a bijection on fibers. A point of the target of (\ref{attracktors to points}) is by definition a $\bG_m$-equivariant map $s: \bA^1 \rightarrow Y$ such that $0 \in \bA^1$ maps to $y_0$. Without loss of generality we can replace $Y$ by this $\bA^1$ and $\cL$ by $s^* \cL$. But now we are done because $s^* \cL$ is trivial as a $\bG_m$-equivariant line bundle on $\bA^1$ and there exists a unique $\bG_m$-equivariant section of $\widetilde{\bA^1} \rightarrow \bA^1$ such that $0 \in \bA^1$ maps to $\widetilde{y}_0$.
	Next, we show that 
	\begin{equation}\label{attractors}
		(\widetilde{Y})^{\leq 0} \cong \Tot(\cL|_{Y^{\leq 0}}) - Y^{\leq 0},
	\end{equation}
	where on the RHS we have the total space with the zero section removed. Note that $$\Tot(\cL|_{Y^{\leq 0}}) - Y^{\leq 0} \cong \widetilde{Y} \times_Y Y^{\leq 0}.$$
	Then for a test affine scheme $S$ a point of this fiber product is the data of 
\begin{equation}\label{diagr}
	\begin{tikzcd}
	S\ar[r, rightarrow, "x_0"']\ar[d, hookrightarrow, "0"']&   	\widetilde{Y}\ar[d, rightarrow, "\pi"']  \\
	\bA^1_S\ar[r, rightarrow, "x"']&   Y
	\end{tikzcd}
\end{equation}
for a $\bG_m$-equivariant $x$. We want to show that there exists a unique $\bG_m$-equivariant $$\widetilde{x}: \bA^1_S \rightarrow \widetilde{Y}$$
	fitting in the diagram (\ref{diagr}). Indeed, note that $x_0$ is $\bG_m$-equivariant. We also know that $x^*\cL$ is trivial as a $\bG_m$-equivariant line bundle, hence there exists a unique $\bG_m$-equivariant section of $x^*\cL$ with prescribed value on $S \xrightarrow{0} \bA^1$, giving the desired $\widetilde{x}$.
	
	Equation (\ref{attracktors to points}) gives that 
	\begin{equation}
		\phi_{f, y_0} \circ i^! \cong 	\phi_{f, y_0}  \circ \widetilde{i}^!\circ\pi^!, 
	\end{equation}
	where $$\widetilde{i}: (\widetilde{Y})^{>0}_{\widetilde{y}_0} \rightarrow \widetilde{Y}.$$
		Equation (\ref{attractors}) gives that $\widetilde{F}|_{\widetilde{Y}^{\leq 0} \setminus \widetilde{y}_0} < 0$, hence by \cite[Theorem 2.2.2] {NT} we get that 
	$$\phi_{f, y_0}  \circ \widetilde{i}^! \xrightarrow{\sim}  \phi_{\widetilde{F}, \widetilde{y}_0}: \Shv_{\widetilde{\Lambda}}^{\bG_m-\eq}(\widetilde{Y}) \rightarrow \Vect,$$
	where $\widetilde{\Lambda} = d \pi (\pi^{-1} \Lambda)$. This concludes construction of (\ref{extension of NT functors}) since $\phi^{\cG}_{F, {y}_0} \cong \phi_{\widetilde{F}, \widetilde{y}_0} \circ \pi^!$. 
	
	We now turn to the second statement of the theorem. But the question is local, so it reduces to the case when the gerbe is trivial, which is \cite[Proposition 2.3.1]{NT}.
 \end{proof}

Combining with Lemma \ref{properties of ms}, we get
\begin{cor}
	The functor $\phi_{f, y_0} \circ i^! $ is t-exact and commutes with Verdier duality up to a shift.
\end{cor}

\subsection{Hyperbolic symmetry.} Fix a $\check{\Lambda}^+$-divisor $D$ of degree $\check{\lambda}$ on $X$ and a closed point $x \in X$ disjoint from the support of $D$. 

\begin{ntn}
	Let $\Bun_{G, N}^{\omega(-D)}(X, nx)$ be the stack classifying $G$-bundles on $X$ with a structure of $B$-reduction on the $n$-th neighborhood $D_{n}(x)$ of $x$, such that the induced $T$-bundle is isomorphic to $\check{\rho}(\omega)(-D)|_{D_{n}(x)}$. 
\end{ntn}

Note that for $n$ large enough the natural map 
$$i: \Bun_{N}^{\omega(-D)}(X) \rightarrow \Bun_G$$
factors through $\Bun_{G, N}^{\omega(-D)}(X, nx).$
\begin{ntn}
For $H = G$, $B$, $B^-$ or $T$, let $\Bun_{H}^{\omega(-D)}(X, nx)$ be the stack classifying $H$-bundles on $X$ with a reduction on the $n$-th order neighborhood $D_{n}(x)$ of $x$ to $\check{\rho}(\omega)(-D)$ via $T \subset H$. 
\end{ntn}

\begin{ntn}
Denote by $\Bun_{N}^{\omega(-D)}(X, nx)$ the fiber product 
\begin{equation}\label{players}
\begin{tikzcd}
\Bun_{N}^{\omega(-D)}(X, nx)\ar[r, rightarrow, "\alpha^{\prime}"]\ar[d, rightarrow, ""']&   	\Bun_{G}^{\omega(-D)}(X, nx)\ar[d, rightarrow, ""']  \\
\Bun_{N}^{\omega(-D)}(X)\ar[r, rightarrow, "\alpha"']&   \Bun_{G, N}^{\omega(-D)}(X, nx).
\end{tikzcd}
\end{equation}
\end{ntn}

\begin{cnstr}\label{action}
	Let $z \in \bG_m$ act on $G$ and $B$ by conjugation by $\check{\rho}(z) \in T$. This gives a $\bG_m$-action on each entree of the diagram (\ref{players}) such that all arrows are $\bG_m$-equivariant.
\end{cnstr}

\begin{ntn}
	Set $$\Bun_{G} \xleftarrow{p} \Bun_B \xrightarrow{q} \Bun_T.$$
\end{ntn}

\begin{ntn}
	Let $U_G \subset \Bun_G$ be a quasi-compact open containing $p(q^{-1} \Bun_T^{(2g-2)\check{\rho}-\deg(D)})$. 
\end{ntn}

By \cite[4.7.2]{DG} there exists a $\bG_m$-invariant open $$W \subset U_G \times_{\Bun_G}  \Bun_{G}^{\omega(-D)}(X, nx)$$ containing the image of $\Bun_{T}^{\omega(-D)}(X, nx)^{(2g-2)\check{\rho}-\deg(D)}$,  such that 
$$W^0 \cong \Bun_{T}^{\omega(-D)}(X, nx)^{(2g-2)\check{\rho}-\deg(D)},$$
$$W^+ \cong \Bun_{B}^{\omega(-D)}(X, nx)^{(2g-2)\check{\rho}-\deg(D)},$$
$$W^- \cong \Bun_{B^-}^{\omega(-D)}(X, nx)^{(2g-2)\check{\rho}-\deg(D)}.$$

\begin{pr}
	For the $\bG_m$-action defined in Construction \ref{action}, 
	$$\Bun_{N}^{\omega(-D)}(X, nx) \rightarrow W$$
	is the attracting locus for $$\check{\rho}(\omega)(-D) \in  \Bun_{T}^{\omega(-D)}(X, nx)^{(2g-2)\check{\rho}-\deg(D)}.$$
\end{pr}

\begin{proof}
	The action of $0 \in \bA^1$ defines a map $q_W: W^+ \rightarrow W^0$, which together with the forgetful map $i_W: W^0 \rightarrow W^+$ defines a retraction 
	\begin{equation}\label{retraction}
		W^0 \xrightarrow{i_W} W^+ \xrightarrow{q_W} W^0.
	\end{equation}
	Then by definition the attracting locus for some $w \in W^0$ is $W^+ \times_{W^0} w$.
	
	It is easy to see that (\ref{retraction}) identifies with 
	\begin{equation}
		\Bun_{T}^{\omega(-D)}(X, nx)^{(2g-2)\check{\rho}-\deg(D)} \rightarrow \Bun_{B}^{\omega(-D)}(X, nx)^{(2g-2)\check{\rho}-\deg(D)} \rightarrow \Bun_{T}^{\omega(-D)}(X, nx)^{(2g-2)\check{\rho}-\deg(D)}
	\end{equation}
	coming from $T \hookrightarrow B$ and $B \twoheadrightarrow T$, hence the result. 
\end{proof}

Recall from \cite[4.1.3]{FGV} that there exists a canonical function
$$\psi_D: \Bun_N^{\omega(-D)} \rightarrow \bA^1.$$

\begin{pr}
	The function $\psi_D$ is $\bG_m$-equivariant for the action on the source defined in Construction \ref{action}, and the action on the target by scaling. 
\end{pr}

\begin{proof}
	To unburden the notation set $\cF_T:= \check{\rho}(\omega)(-D).$ For $$\check{\lambda}: T \rightarrow \bG_m$$
	denote by $\cL_{\cF_T}^{\check{\lambda}}$ the $\bG_m$-bundle induced from $\cF_T$. By definition, $\psi_D$ is the sum of the functions of the form
	\begin{equation}\label{char}
		\Bun_N^{\cF_T} \rightarrow \Bun_{\bG_a}^{\cL_{\cF_T}^{\check{\alpha}_i}, \cO} \cong H^1(X, \cL_{\cF_T}^{\check{\alpha}_i}) \times BH^0(X, \cL_{\cF_T}^{\check{\alpha}_i}) \rightarrow H^1(X, \cL_{\cF_T}^{\check{\alpha}_i}) \rightarrow H^1(X, \omega) \cong \bG_a.
	\end{equation}
	Let us check that each of these functions is $\bG_m$-equivariant. Indeed, the first map in (\ref{char}) is induced by a map $$B \rightarrow B_{\GL_2} \subset \Aut(\cV^i),$$
	which is $\bG_m$-equivariant for the action of $\bG_m$ given in Construction \ref{action}. Other maps in (\ref{char}) are $\bG_m$-equivariant by construction, with the action of $g \in \bG_m$ on $H^i(X, \cL_{\cF_T}^{\check{\alpha}_i})$ and $H^1(X, \omega)$ given by multiplying a class by $g$. 
	
\end{proof}

\subsection{Twisted Whittaker coefficients and microstalks.}
We want to apply Theorem \ref{extension of NT} to the case where $Y = W$, the $\bG_m$-action is as in Construction \ref{action}, and $\cL$ is pullback of the determinant line bundle on $\Bun_G$ (note that $\cL$ is $\bG_m$-equivariant since conjugation by  $\check{\rho}(z) \in T/Z(G)$ gives trivial action on $\Bun_G$). We take $\Lambda = \underline{\Nilp}_{\redu}^{\check{\lambda}}$ (see Notation \ref{Nilpgeq}). Let us check that conditions of Theorem \ref{extension of NT} are satisfied. 

\begin{lm}\label{W}
	The diagram 	\[
	\begin{tikzcd}
	(\widetilde{W})^{0}\ar[r, rightarrow, ""']\ar[d, rightarrow, ""']&   	\widetilde{W}\ar[d, rightarrow, "\pi"']  \\
	W^{0}\ar[r, rightarrow, ""']&   W
	\end{tikzcd}
	\]
	is Cartesian.
\end{lm}

\begin{proof}
	Since the action of $\bG_m$ preserves the zero section of $\cL$ it suffices to check that the diagram 
	\[
	\begin{tikzcd}
	\Tot(\cL)^{0}\ar[r, rightarrow, ""']\ar[d, rightarrow, ""']&   	\Tot(\cL)\ar[d, rightarrow, ""']  \\
	W^{0}\ar[r, rightarrow, ""']&   W
	\end{tikzcd}
	\]
	is Cartesian.
	
	Indeed, for a test scheme $S$ a map $S \rightarrow 	\Tot(\cL)^{0}$ is by definition a $\bG_m$-equivariant map $S \rightarrow 	\Tot(\cL)$. The latter gives via composition a $\bG_m$-equivariant map $S \rightarrow W$, which by definition factors through $W^0$. Hence giving a map $S \rightarrow 	\Tot(\cL)^{0}$ is equivalent to giving a $\bG_m$-equivariant map $S \rightarrow \Tot(\cL|_{W^0})$. It is left to notice that the action of $\bG_m$ on $\Tot(\cL|_{W^0})$ is trivial.
\end{proof}


Let us consider a restriction of the diagram (\ref{players}) to $U_G$, i.e. 
\begin{equation}\label{players2}
\begin{tikzcd}
\Bun_{N}^{\omega(-D)}(X, nx)\ar[r, rightarrow, "\alpha^{\prime}_{U_G}"]\ar[d, rightarrow, ""']&   	\Bun_{G}^{\omega(-D)}(X, nx)|_{U_G}\ar[d, rightarrow, ""']  \\
\Bun_{N}^{\omega(-D)}(X)\ar[r, rightarrow, "\alpha_{U_G}"']&   \Bun_{G, N}^{\omega(-D)}(X, nx)|_{U_G}.
\end{tikzcd}
\end{equation}

As explained in \cite[(3.12)]{NT} the map $\alpha^{\prime}_{U_G}$ is a locally closed embedding. Therefore $\alpha_{U_G}$ is a locally closed embedding as well, since the downward arrows in (\ref{players2}) are surjective.

\begin{pr}\label{transversal intersection 1}
	Inside $T^*\Bun_G$ the shifted conormal bundle $T^*_{\Bun_N^{\omega(-D)}(X)} \Bun_G(X) + d\psi_D$ intersects $\Lambda$ transversely at a smooth point.
\end{pr}

\begin{proof}
	Let us check that the shifted conormal bundle $$T^*_{\Bun_N^{\omega(-D)}(X)} \Bun_G(X) + d\psi_D := \Bun_N^{\omega(-D)}(X) \times_{T^* \Bun_N^{\omega(-D)}(X)} (T^*\Bun_G \times_{\Bun_G} \Bun_N^{\omega(-D)}(X))$$
	coincides with $\Kos_D$. For simplicity we will check that they coincide on the level of $k$-points, keeping in mind that the same argument works for $S$-points for any affine $S$. 
	
	We described points of $\Kos_D$ in Remark \ref{kosD points}. Let us describe points of the shifted conormal bundle. 
	By definition, $\Bun_N^{\omega(-D)}(X) $ parameterizes $B$-bundles $\cF_B$ such that induced $T$-bundle is $\check{\rho}(\omega)(-D)$. The stack $T^* \Bun_N^{\omega(-D)}(X)$ parameterizes such $B$-bundles $\cF_B$ and an element in $\Gamma(X, \omega \otimes \cF_B \times^B \mathfrak{n}^*)$. The map
	$$d\psi_D: \Bun_N^{\omega(-D)}(X) \rightarrow T^* \Bun_N^{\omega(-D)}(X)$$
	is given by $$\cF_B \mapsto (\cF_B, t),$$ where $t$ is the image of the canonical section $s_0$ under 
	\begin{equation}\label{section canonical}
		\Gamma(X, \oplus_{i \in \cI} \cO(\check{\alpha}_i(D))) \rightarrow \Gamma(X, \omega \otimes \cF_T \times^T (\mathfrak{n}/[\mathfrak{n}, \mathfrak{n}])^*) \rightarrow \Gamma(X, \omega \otimes \cF_B \times^B \mathfrak{n}^*).
	\end{equation}
	Finally, the stack $T^*\Bun_G \times_{\Bun_G} \Bun_N^{\omega(-D)}(X)$ parameterizes $B$-bundles $\cF_B$ with the prescribed induced $T$-bundle and a section in $\Gamma(X, \omega \otimes (\cF_B \times^B \mathfrak{g}^*))$. Note that the map 
	$$T^*\Bun_G \times_{\Bun_G} \Bun_N^{\omega(-D)}(X) \rightarrow T^* \Bun_N^{\omega(-D)}(X) $$ is induced by $\mathfrak{g}^* \rightarrow \mathfrak{n}^*$.
	Identifying $\mathfrak{g}^* \cong \mathfrak{g}$ we get that points of the shifted conormal bundle are given by 
	\begin{itemize}
		\item a $B$-bundle $\cF_B$ on $X$ such that the induced $T$-bundle $\cF_T \cong \check{\rho}(\omega)(-D)$,
		\item an element $s \in \Gamma(X, \omega \otimes \cF_B \times^B \mathfrak{g})$, such that its image in $\Gamma(X, \omega \otimes \cF_B \times^B \mathfrak{g}/\mathfrak{b})$ identifies with the image of $s_0$ under
		$$\Gamma(X, \oplus_{i \in \cI} \cO(\check{\alpha}_i(D))) \rightarrow \Gamma(X, \omega \otimes \cF_T \times^T \mathfrak{g}_{\geq -1}/\mathfrak{b}) \rightarrow \Gamma(X, \omega \otimes \cF_B \times^B \mathfrak{g}/\mathfrak{b}).$$
	\end{itemize}
	This coincides with the description in Remark \ref{kosD points}. Therefore the intersection is the point $\lambda_D$ by Lemma \ref{non-empty intersections}, and $\lambda_D \in \Lambda$ is smooth by Lemma \ref{nilpintersections}.
\end{proof}

\begin{pr}\label{transversal intersection 2}
	Inside $T^*\Bun_{G, N}^{\omega(-D)}(X, nx)|_{U_G}$ the shifted conormal $$\Lambda^{''}_{\psi_D}:=T^*_{\Bun_N^{\omega(-D)}(X)} \Bun_{G, N}^{\omega(-D)}(X, nx)|_{U_G}+ d\psi_D$$
	intersects $\Lambda \times_{\Bun_G(X)} \Bun_{G, N}^{\omega(-D)}(X, nx)|_{U_G}$ transversely at a smooth point.
\end{pr}

\begin{proof}
	Denote $$\Bun_{G, N}^{\omega(-D)}(X, nx)|_{U_G}$$ by $V$. Tautologically, we have 
	$$\Lambda \times_{\Bun_G(X)}V \subset T^*\Bun_G(X) \times_{\Bun_G(X)} V.$$
	Proposition \ref{transversal intersection 1} implies that $\Lambda \times_{\Bun_G(X)} V$ intersects 
	$$\Lambda^{''}_{\psi_D} \cap (T^*\Bun_G(X) \times_{\Bun_G(X)} V) \cong T^*_{\Bun_N^{\omega(-D)}(X)} \Bun_{G}(X) + d\psi_D$$
	inside of $T^*\Bun_G(X) \times_{\Bun_G(X)} V$ transversely at the smooth point $$\lambda_D^{\prime}:= \{\lambda_D, \check{\rho}(\omega)(-D)^{\prime}\},$$
	where $$\check{\rho}(\omega)(-D)^{\prime}:=(\check{\rho}(\omega)(-D), (\check{\rho}(\omega)(-D)\times^T B) |_{D_n(x)}) \in V.$$Hence it suffices to show that $$T^*\Bun_G(X) \times_{\Bun_G(X)} V$$ intersects $\Lambda^{''}_{\psi_D}$ inside $T^*\Bun_{G, N}^{\omega(-D)}(X, nx)|_{U_G}$ transversely at 
	$$\lambda_D^{''}:= \{\check{\rho}(\omega)(-D)^{\prime}, \sigma\} \in  T^*V,$$
	where $$\sigma \in \Gamma(X, \omega(nx) \otimes (\check{\rho}(\omega)(-D) \times^T \mathfrak{g}^*))$$ is the image of a canonical element in $\Gamma(X, \omega(nx) \otimes (\check{\rho}(\omega)(-D) \times^T \mathfrak{n}^*))$.
	 Indeed, the argument is the same as in \cite[Proposition 3.3.2]{NT}, but we include it for completeness.

	First, note that $T_{\lambda_D^{''}} T^*V$ fits into a short exact sequence
	\begin{equation}
		T^*_{\check{\rho}(\omega)(-D)^{\prime}} V \rightarrow T_{\lambda_D^{''}} T^*V \rightarrow T_{\check{\rho}(\omega)(-D)^{\prime}} V.
	\end{equation}
	However, $$T_{\lambda_D^{\prime}} (T^*\Bun_G(X) \times_{\Bun_G(X)} V)$$
	fits into a short exact sequence
	\begin{equation}
		T^*_{\check{\rho}(\omega)(-D)} \Bun_G(X) \rightarrow T_{\lambda_D^{\prime}} (T^*\Bun_G(X) \times_{\Bun_G(X)} V)\rightarrow T_{\check{\rho}(\omega)(-D)^{\prime}} V.
	\end{equation}
	Therefore it suffices to show that $$T_{\lambda_D^{\prime}} (T^*\Bun_G(X) \times_{\Bun_G(X)} V)$$ and $T_{\lambda_D^{''}}\Lambda^{''}_{\psi_D}$ span the image of $T^*_{\check{\rho}(\omega)(-D)^{\prime}} V$ inside $T_{\lambda_D^{''}} T^*V$. 
	The cotangent space $T^*_{\check{\rho}(\omega)(-D)^{\prime}} V$ fits into a short exact sequence
	\begin{equation}
		(T^*_{\Bun_{N}^{\omega(-D)}} V)_{\check{\rho}(\omega)(-D)}\rightarrow T^*_{\check{\rho}(\omega)(-D)^{\prime}} V\rightarrow T^*_{\check{\rho}(\omega)(-D)} \Bun_{N}^{\omega(-D)}.
	\end{equation}
	Since $\check{\rho}(\omega)(-D) \in \Bun_{N}^{\omega(-D)}$ was induced from a $T$-bundle, the codifferential 
	$$T^*_{\check{\rho}(\omega)(-D)} \Bun_G(X) \twoheadrightarrow T^*_{\check{\rho}(\omega)(-D)} \Bun_{N}^{\omega(-D)}$$ is surjective. Finally, note that 
	
	\begin{equation}
	(T^*_{\Bun_{N}^{\omega(-D)}} V)_{\check{\rho}(\omega)(-D)}
	\cong  T^*_{\check{\rho}(\omega)(-D)^{\prime}} V\times_{T_{\lambda_D^{''}} (T^*V\times_{V}  \Bun_{N}^{\omega(-D)})} T_{\lambda_D^{''}}\Lambda^{''}_{\psi_D} \hookrightarrow  T_{\lambda_D^{''}}\Lambda^{''}_{\psi_D} .
	\end{equation}
\end{proof}

\begin{cor}\label{clean intersection}
	Inside $T^*\Bun_{G}^{\omega(-D)}(X, nx)$ the shifted conormal bundle
	$$T^*_{\Bun_{N}^{\omega(-D)}(X, nx)}\Bun_{G}^{\omega(-D)}(X, nx)+ d\psi_D^{\prime}$$
	intersects $\Lambda \times_{\Bun_{G}(X)} \Bun_G^{\omega(-D)}(X, nx)$ cleanly along smooth points. The intersection is $n \dim(N)$-dimensional.
\end{cor}

\begin{proof}
	Follows from Proposition \ref{transversal intersection 2} as in \cite[Proposition 3.3.3]{NT}.
\end{proof}

\begin{ntn}
	Denote $$\Nilp_D:= \Nilp^{\reg, (2-2g)\check{\rho} + \deg(D)}_{\redu}.$$
\end{ntn}

We are now ready to state the main result of this section. 

\begin{thm}\label{whittaker functional is microstalk}
	For $\cF \in \Shv_{\kappa, \underline{\Nilp}_{\redu}^{\check{\lambda}}}(\Bun_G)$ and a $\check{\Lambda}^+$-valued divisor $D$ with $$\check{\lambda} = (2-2g)\check{\rho} + \deg(D),$$ the shifted twisted Whittaker functional 
	$$\coeff_D(\cF):= \phi_{\psi_D, \check{\rho}(\omega)(-D)} i^!(\cF) [\dim_{\check{\rho}(\omega)(-D)}  \Bun_{B^-}(X)]$$
	calculates microstalk. In particular, it is t-exact, commutes with Verdier duality, and 
	$$\chi(\coeff_D(\cF)) = c_{\Nilp_D, \cF}.$$
\end{thm}

\begin{proof}
	By Lemma \ref{W}, Corollary \ref{clean intersection}, and Theorem \ref{extension of NT} there exists a real-valued function $F$ on $W$ such that 
	$$\phi_{\psi_D^{\prime}, \check{\rho}(\omega)(-D)} ({{\alpha}^{\prime}_W})^! \cong \phi_{F, \check{\rho}(\omega)(-D)}: \Shv_{\cL^a, \Lambda}(W) \rightarrow \Vect$$
	is a microstalk shifted by $\dim_{\check{\rho}(\omega)(-D)} \Bun_{B^-}(X, nx) - n\dim N$. Therefore $$\phi_{\psi_D^{\prime}, \check{\rho}(\omega)(-D)} ({{\alpha}^{\prime}_W})^![\dim_{\check{\rho}(\omega)(-D)} \Bun_{B^-}(X, nx) - n\dim N]$$ is t-exact and commutes with Verdier duality.
	But we have 
	$$\phi_{\psi_D^{\prime}, \check{\rho}(\omega)(-D)}({{\alpha}^{\prime}_W})^! j_W^!\pi^![-n \dim G]\cong  \phi_{\psi_D, \check{\rho}(\omega)(-D)} i^! [2n\dim N - n \dim G],$$
	and by smoothness $j_W^! \circ \pi^![-n \dim G]$ is exact and commutes with Verdier duality, therefore 
	$$  \phi_{\psi_D, \check{\rho}(\omega)(-D)} i^! [\dim_{\check{\rho}(\omega)(-D)} \Bun_{B^-}(X) ]$$
	is exact and commutes with Verdier duality. Here $$j_W: W \hookrightarrow \Bun_G^{\omega(-D)}(X, nx)$$ and $$\pi: \Bun_G^{\omega(-D)}(X, nx) \rightarrow \Bun_G.$$
	
	Finally, by Proposition \ref{microstalk_cc} we have 
	\begin{equation}\label{cc}
		\chi(\phi_{\psi_D^{\prime}, \check{\rho}(\omega)(-D)} \circ ({{\alpha}^{\prime}_W})^! (j_W^!\pi^!\cF)[\dim_{\check{\rho}(\omega)(-D)} \Bun_{B^-}(X, nx) - 2n\dim N]) \cong c_{\widetilde{\Nilp_D},(j_W^!\pi^!\cF)[-n\dim N] },
	\end{equation}
	where $$\widetilde{\Nilp_D}:= \Nilp_D \times_{T^*\Bun_G} (W \times_{\Bun_G} T^*\Bun_G) \hookrightarrow T^*W.$$
	However, the left side of (\ref{cc}) coincides with 
	$$\chi(\coeff_D(\cF)),$$
	and the right side of (\ref{cc}) coincides with $c_{\Nilp_D, \cF}$ by \cite[Proposition 9.4.3]{KS}.

\end{proof}

\section{Conservativity of Whitttaker coefficients.}

\subsection{Conservativity for nilpotent sheaves.}

\begin{thm}\label{conservativity nilp sheaves}
	The functor $$\coeff^{\loc}: \Shv_{\kappa, \Nilp}(\Bun_{G})^{\temp} \rightarrow \Whit_{\kappa}(\Gra_G)$$ is conservative.
\end{thm}

\begin{rem}
	More generally, this implies that for every $\cF \in \Shv_{\kappa, \Nilp}(\Bun_{G})$ with non-zero projection to $ \Shv_{\kappa, \Nilp}(\Bun_{G})^{\temp}$ we have $$\coeff^{\loc}(\cF) \neq 0.$$
\end{rem}

\begin{proof}
	We have $\cF \notin \Shv_{\kappa, \Nilp}(\Bun_{G})^{\atemp}$, therefore by Proposition \ref{temperedSS} we have  $$\SingS(\cF) \cap \Nilp_{\redu}^{\reg} \neq 0.$$
	Choose a minimal $\check{\lambda} \in \check{\Lambda}^{\rel}$ such that $\SingS(\cF) \cap \Nilp^{\reg, \check{\lambda}} \neq 0$ and a $\check{\Lambda}^+$-valued divisor $D$ on $X$ with $$\check{\lambda} = (2-2g)\check{\rho} + \deg(D).$$ We claim that $\coeff_D(\cF) \neq 0$.   
	
	Indeed, there exists $i$ and a constructible $\cG \in H^i(\cF)$ with $\Nilp^{\reg, \check{\lambda}}  \in \SingS(\cG)$. Then by Theorem \ref{whittaker functional is microstalk} we have $\coeff_D(\cG) \neq 0$. By Theorem \ref{whittaker functional is microstalk} we have 
	$$H^0(\coeff_D(\cG)) \hookrightarrow H^0(\coeff_D(H^i(\cF))) \cong H^i(\coeff_D(\cF)),$$
	which implies the claim.
\end{proof}

\subsection{Conjectural metaplectic spectral action.}\label{SA}

Recall from \cite[6.3.3]{GL} that to $\kappa$ we can associate a metaplectic Langlands dual datum
$$(H, \cG_Z, \epsilon)$$
consisting of a reductive group $H$, a $Z$-gerbe  $\cG_Z$ on $X$ (where $Z$ is the center of $H$), and a homomorphism 
$$\epsilon: \bZ / 2\bZ \rightarrow Z(k).$$

\subsubsection{} In \cite[9.2]{GL}, Gaitsgory and Lysenko construct the factorizable metaplectic geometric Satake functor 
\begin{equation}
	\Sat: \Rep(H)_{\cG_Z, \Ranp}^{\epsilon} \rightarrow \Sph_{\kappa, \Ranp}
\end{equation}
from an appropriately twisted factorization category of representations of $H$ to twisted spherical category. This defines an action of $\Rep(H)_{\cG_Z, \Ranp}^{\epsilon}$ on $D_{\kappa}(\Bun_G)$. Parallel to the untwisted case, Gaitsgory and Lysenko (\cite[9.6]{GL}) construct a spectral localization symmetric monoidal functor 
\begin{equation}
	\Loc: \Rep(H)_{\cG_Z, \Ranp}^{\epsilon} \rightarrow \QCoh(\LocSys_H^{\cG_Z})
\end{equation}
to the category of quasi-coherent sheaves on the stack of twisted local systems (\cite[8.4]{GL}).

\begin{conj}\label{spectral action}
	The action of $\Rep(H)_{\cG_Z, \Ranp}^{\epsilon}$ on $D_{\kappa}(\Bun_G)$ uniquely factors through $\QCoh(\LocSys_H^{\cG_Z})$ via $\Loc$.
\end{conj}

\subsection{Conservativity for general D-modules.}
 In this subsection we will prove the following statement modulo Conjecture \ref{spectral action}:
 
 \begin{thm}\label{general}
 	The functor 
 	$$\coeff^{\loc}: D_{\kappa}(\Bun_{G})^{\temp} \rightarrow \Whit_{\kappa}(\Gra_G)$$ is conservative.
 \end{thm}

The proof will follow the argument in \cite[10.3.3]{FR}. First, we generalize \cite[20.9.1]{AGKRRV}.

\begin{lm}\label{20.9.1}
	 For any field extension $k^\prime / k$, the inclusion
	 \begin{equation}
	 \begin{split}
	  \Shv_{\kappa, /k^{\prime}, \Nilp_{k^{\prime}}}(\Bun_{G, k^{\prime}}) \otimes_{\QCoh(\LocSys_{H, k^{\prime}}^{\cG_Z})}\Vect_{\sigma^{\prime}} \rightarrow D_{\kappa,  /k^{\prime}}(\Bun_{G,  k^{\prime}}) \otimes_{\QCoh(\LocSys_{H,  k^{\prime}}^{\cG_Z})}\Vect_{\sigma^{\prime}} \cong \\
	  \cong D_{\kappa}(\Bun_{G}) \otimes_{\QCoh(\LocSys_H^{\cG_Z})}\Vect_{\sigma^{\prime}} 
	 \end{split}
	 \end{equation}
	and its analogue for the tempered subcategories are equivalences.
\end{lm}

We will also use the folliwing statement:
\begin{lm}\label{Rep(H) and LS}
	The natural functor
	$$\QCoh(\LocSys_H^{\cG_Z}) \otimes_{\Rep(H)_{\cG_Z, \Ranp}^{\epsilon}} \QCoh(\LocSys_H^{\cG_Z}) \rightarrow \QCoh(\LocSys_H^{\cG_Z})$$
	is an equivalence of categories.
\end{lm}

We postpone the proofs of Lemma \ref{20.9.1} and Lemma \ref{Rep(H) and LS} until section \ref{spectral proj section}.

\begin{proof}[Proof of Theorem \ref{general}]
	Suppose we are given $\cF \in D_{\kappa}(\Bun_{G})$ with $\coeff^{\loc}(\cF) = 0$. We want to show that the image of $\cF$ in $D_{\kappa}(\Bun_{G})^{\temp}$ is zero. 
	
	By \cite[Lemma 21.4.6]{AGKRRV}, it suffices to show that for any field extension $k^{\prime} / k$ and any $\sigma^{\prime} \in \LocSys_H^{\cG_Z}(k^{\prime})$ the image of $\cF$ in $$D_{\kappa}(\Bun_{G})^{\temp} \otimes_{\QCoh(\LocSys_H^{\cG_Z})} \Vect_{\sigma^{\prime}}$$ is zero. 
	We also have a commutative diagram (we implicitly use Lemma \ref{Rep(H) and LS})
	\begin{equation}\label{AGKRRV}
	\begin{tikzcd}
D_{\kappa}(\Bun_{G})^{\temp} \otimes_{\QCoh(\LocSys_H^{\cG_Z})}\Vect_{\sigma^{\prime}}\ar[dr, rightarrow, ""]&   	  \\
	 \Shv_{\kappa, /k^{\prime}, \Nilp_{k^{\prime}}}(\Bun_{G, k^{\prime}})^{\temp} \otimes_{\QCoh(\LocSys_{H, k^{\prime}}^{\cG_Z})}\Vect_{\sigma^{\prime}}\ar[u, rightarrow, "\simeq"']\ar[r, rightarrow, ""']&   \Whit_{\kappa}(\Gra_G) \otimes_{\Rep(H)_{\cG_Z, \Ranp}^{\epsilon}} \Vect_{\sigma^{\prime}}.
	\end{tikzcd}
	\end{equation}
	Here the vertical arrow is an equivalence by Lemma \ref{20.9.1}. By Theorem \ref{conservativity nilp sheaves} the horizontal arrow in (\ref{AGKRRV}) is conservative, hence the diagonal arrow is conservative, and therefore $\cF$ is zero.
\end{proof}

\subsection{Metaplectic Beilinson's spectral projector.}\label{spectral proj section}
Recall that Beilinson’s projector is a functor that manifactures Hecke
eigen-objects (in the sense of \cite[9.5.3]{GL}) from arbitrary objects of $\D_{\kappa}(\Bun_G)$. 

\begin{ntn}
	For a point $\sigma \in \LocSys_H^{\cG_Z}$ set $$\operatorname{Hecke}_{\sigma}:= \D_{\kappa}(\Bun_G) \otimes_{\Rep(H)_{\cG_Z, \Ranp}^{\epsilon}} \Vect_{\sigma},$$
	where the action of $\Rep(H)_{\cG_Z, \Ranp}^{\epsilon}$ on $\Vect_{\sigma}$ comes from $\operatorname{Ev}_{\sigma}$ (\cite[9.5.3]{GL}). We will call elements of $\operatorname{Hecke}_{\sigma}$ \emph{Hecke eigensheaves}.
\end{ntn}

We can also write
\begin{equation}
\operatorname{Hecke}_{\sigma} = \Maps_{\Rep(H)_{\cG_Z, \Ranp}^{\epsilon} \otimes \Rep(H)_{\cG_Z, \Ranp}^{\epsilon} \operatorname{-comod}}(\Rep(H)_{\cG_Z, \Ranp}^{\epsilon}, \D_{\kappa}(\Bun_G))
\end{equation}
\begin{pr}\label{rigid}
	The category  $\Rep(H)_{\cG_Z, \Ranp}^{\epsilon}$ is rigid.
\end{pr}

\begin{df}
	Let $\mathbf{C}$ be a symmetric monoidal $\D(X)$-linear category, $\mathbf{H}$ be a dualizable symmetric monoidal $\D(X)$-linear category. Parallel to \cite[8.2]{AGKRRV}, there exists a symmetric monoidal category $$\underline{\operatorname{coHom}}^{\D(X)}(\mathbf{C}, \mathbf{H}),$$
	defined by the universal property that for a target symmetric monoidal category $\mathbf{A}$ we have 
$$\Maps_{\DGCat^{\operatorname{SymMon}}}(\underline{\operatorname{coHom}}^{\D(X)}(\mathbf{C}, \mathbf{H}), \mathbf{A}) \cong \Maps_{\D(X)\bfmod^{\operatorname{SymMon}}}(\mathbf{C}, \mathbf{A}\otimes \mathbf{H}).$$
\end{df}

\begin{rem}\cite[11.1.9]{AGKRRV}\label{coHom}
	The symmetric monoidal functor 
	$$\Rep(H)_{\cG_Z, \Ranp}^{\epsilon} \rightarrow \underline{\operatorname{coHom}}^{\D(X)}(\Rep(H)_{\cG_Z}^{\epsilon}(X), \D(X))$$
	is an equivalence.
\end{rem}

\begin{proof}[Proof of Proposition \ref{rigid}]
	Note that multiplication map
	\begin{equation}
		m_{\Ranp}: \underline{\operatorname{coHom}}^{\D(X)}(\Rep(H)_{\cG_Z}^{\epsilon}(X), \D(X)) ^{\otimes 2}  \rightarrow \underline{\operatorname{coHom}}^{\D(X)}(\Rep(H)_{\cG_Z}^{\epsilon}(X), \D(X))
	\end{equation}
	identifies with the map
	\begin{equation}
	\begin{split}
	m_{\Rep(H)_{\cG_Z}^{\epsilon}(X)}: \underline{\operatorname{coHom}}^{\D(X)}\left(\Rep(H)_{\cG_Z}^{\epsilon}(X) \otimes_{\D(X)} \Rep(H)_{\cG_Z}^{\epsilon}(X), \D(X)\right)  \rightarrow \\
	\rightarrow \underline{\operatorname{coHom}}^{\D(X)}\left(\Rep(H)_{\cG_Z}^{\epsilon}(X), \D(X)\right)  
	\end{split}
	\end{equation}
	induced from the monoidal structure on $\Rep(H)^{\epsilon}$. Therefore, since $\Rep(H)^{\epsilon}$ is rigid, we get that $m_{\Ranp}$ admits a continuous right adjoint compatible with the action of $\underline{\operatorname{coHom}}^{\D(X)}(\Rep(H)_{\cG_Z}^{\epsilon}(X), \D(X)) $.
	
	To show that the unit of $\Rep(H)_{\cG_Z, \Ranp}^{\epsilon}$ is compact, recall that 
	\begin{equation}\label{Ran colim}
		 \Rep(H)_{\cG_Z, \Ranp}^{\epsilon} \cong \colim\limits_{(I \rightarrow J) \in \operatorname{Tw}}\bigotimes_{j \in J} (\Rep(H)_{\cG_Z}^{\epsilon})^{\otimes I_j}(X),
	\end{equation}
	where $\operatorname{Tw}$ is the twisted arrow category and $$(\Rep(H)_{\cG_Z}^{\epsilon})^{\otimes I_j}(X)$$
	is the $I_j$-fold tensor product over $D(X)$. Since all connecting arrows in the colimit diagram of (\ref{Ran colim}) admit continuous right adjoints, all insertion functors also admit continuous right adjoints. Thus the unit object of $\Rep(H)_{\cG_Z, \Ranp}^{\epsilon}$ is compact.
\end{proof}

The adjunction
\begin{equation}
\begin{tikzcd}
{\operatorname{mult}}: \Rep(H)_{\cG_Z, \Ranp}^{\epsilon} \otimes \Rep(H)_{\cG_Z, \Ranp}^{\epsilon} \ar[r, rightarrow,  shift right, ""]& \arrow[l,  shift right, ""]\Rep(H)_{\cG_Z, \Ranp}^{\epsilon}: {\operatorname{comult}}
\end{tikzcd}
\end{equation}
as $\Rep(H)_{\cG_Z, \Ranp}^{\epsilon}$-bimodule categories induces
\begin{equation}
\begin{tikzcd}
\mathbf{ind}_{\operatorname{Hecke}}: \D_{\kappa}(\Bun_G) \ar[r, rightarrow,  shift right, ""]& \arrow[l,  shift right, ""]\operatorname{Hecke}_{\sigma}: \mathbf{oblv}_{\operatorname{Hecke}}
\end{tikzcd}
\end{equation}
\begin{ntn}
	We refer to $P_{\sigma}:= \mathbf{ind}_{\operatorname{Hecke}}$ as \emph{metaplectic Beilinson's spectral projector}.
\end{ntn}

\subsubsection{} Note that  $\mathbf{oblv}_{\operatorname{Hecke}}$ is conservative, and hence monadic. The monad is given by the action of 
\begin{equation}
	R_{\sigma}:= (\Id \otimes \operatorname{Ev}_{\sigma})(R_{\Rep(H)_{\cG_Z, \Ranp}^{\epsilon}}) \in \Rep(H)_{\cG_Z, \Ranp}^{\epsilon},
\end{equation}
where $$R_{\Rep(H)_{\cG_Z, \Ranp}^{\epsilon}}\in \Rep(H)_{\cG_Z, \Ranp}^{\epsilon} \otimes \Rep(H)_{\cG_Z, \Ranp}^{\epsilon}$$ is the unit of the self-duality.

Adapting the proof of \cite[Main Theorem 14.4.4]{AGKRRV}, we get 
\begin{pr}\cite[Corollary 14.4.10]{AGKRRV}\label{nilpss}
	Hecke eigensheaves in $\Shv_{\kappa}(\Bun_G) \subset \D_{\kappa}(\Bun_G)$ have nilpotent singular support.
\end{pr}

\subsubsection{} As in \cite[16.2.1]{AGKRRV} we can take a collection of $k$-points $y_i$ on $\Bun_G$ (finitely many in every quasi-compact open) such that for every non-zero $\cF \in \Shv_{\kappa, \Nilp}(\Bun_G)$, the $!$-fiber of $\cF$ at least one point $y_i$ will be non-zero. 

\begin{lm}\label{gen}
	The objects $P_{\sigma}(\delta_{y_i})$ generate $$\Shv_{\kappa, \Nilp}(\Bun_G) \otimes_{\QCoh(\LocSys_H^{\cG_Z})} \Vect_{\sigma}.$$
\end{lm}
\begin{proof}
	For $$\cF \in \Shv_{\kappa, \Nilp}(\Bun_G) \otimes_{\QCoh(\LocSys_H^{\cG_Z})} \Vect_{\sigma}$$ we have 
	\begin{equation}
		\begin{split}
		\Maps_{ \Shv_{\kappa, \Nilp}(\Bun_G) \otimes_{\QCoh(\LocSys_H^{\cG_Z})} \Vect_{\sigma}}(P_{\sigma}(\delta_{y_i}), \cF) \cong \Maps_{ \D_{\kappa}(\Bun_G) \otimes_{\QCoh(\LocSys_H^{\cG_Z})} \Vect_{\sigma}}(P_{\sigma}(\delta_{y_i}), \cF) \cong \\
		\cong  \Maps_{\D_{\kappa}(\Bun_G)}(\delta_{y_i}, \mathbf{oblv}_{\operatorname{Hecke}}(\cF)) \cong \Maps_{\Shv_{\kappa}(\Bun_G)}(\delta_{y_i}, \mathbf{oblv}_{\operatorname{Hecke}}(\cF)).
		\end{split}
	\end{equation}
	Since $\mathbf{oblv}_{\operatorname{Hecke}}$ is conservative we get the result.
\end{proof}

\begin{proof}[Proof of Lemma \ref{20.9.1}]
	Using Proposition \ref{nilpss} and Lemma \ref{gen}, the argument in \cite[20.9.4]{AGKRRV} proves Lemma \ref{20.9.1}.
\end{proof}

\begin{proof}[Proof of Lemma \ref{Rep(H) and LS}]
	Using Remark \ref{coHom}, the functor $\Loc$ identifies with \cite[(12.12)]{AGKRRV}. Then the proof follows \cite[Theorem 12.6.3]{AGKRRV}.
\end{proof}

\bibliographystyle{alpha}
\bibliography{QCOEFFSbibliography}

\end{document}